\documentclass[11pt]{article}

\usepackage{latexsym}
\usepackage{amsxtra}
\usepackage{amssymb}
\usepackage{bm}
\usepackage{amsthm}
\usepackage{mathtools}
\usepackage{xspace}
\usepackage{graphicx}
\usepackage[margin=1in]{geometry}
\usepackage{xfrac}

\input xy
\xyoption{all} \CompileMatrices

\begin{document}

\def\HH{{\mathcal{H}}}
\def\orb{{\operatorname{orb}}}
\def\diam{{\operatorname{diam}}}
\def\II{{\mathfrak{I}}}
\def\PO{{\operatorname{PO}}}
\def\Cl{{\operatorname{Cl}}}
\def\Max{{\operatorname{-Max}}}
\def\XX{{\mathfrak X}}
\def\YY{{\bf{Y}}}
\def\BBB{{\mathcal B}}
\def\inv{{\operatorname{inv}}}
\def\emph{\it}
\def\Int{{\operatorname{Int}}}
\def\Spec{\operatorname{Spec}}
\def\Bin{{\operatorname{B}}}
\def\n{\operatorname{b}}
\def\N{{\operatorname{GB}}}
\def\BC{{\operatorname{BC}}}
\def\dlog{\frac{d \log}{dT}}
\def\Sym{\operatorname{Sym}}
\def\Nr{\operatorname{Nr}}
\def\lbrack{{\{}}
\def\rbrack{{\}}}
\def\burnside{\operatorname{B}}
\def\Sym{\operatorname{Sym}}
\def\Hom{\operatorname{Hom}}
\def\Inj{\operatorname{Inj}}
\def\Aut{{\operatorname{Aut}}}
\def\Mor{{\operatorname{Mor}}}
\def\Map{{\operatorname{Map}}}
\def\CMap{{\operatorname{CMap}}}
\def\GMaps{G{\operatorname{-Maps}}}
\def\Fix{{\operatorname{Fix}}}
\def\res{{\operatorname{res}}}
\def\ind{{\operatorname{ind}}}
\def\inc{{\operatorname{inc}}}
\def\coind{{\operatorname{cnd}}}
\def\Equiv{{\mathcal{E}}}
\def\W{\operatorname{W}}
\def\F{\operatorname{F}}
\def\witt{\operatorname{gh}}
\def\ngh{\operatorname{ngh}}
\def\Fm{{\operatorname{Fm}}}
\def\bij{{\iota}}
\def\mk{{\operatorname{mk}}}
\def\km{{\operatorname{mk}}}
\def\VV{{\bf{V}}}
\def\ff{{\bf{f}}}
\def\ZZ{{\mathbb Z}}
\def\Zhat{{\widehat{\mathbb Z}}}
\def\CC{{\mathbb C}}
\def\PP{{\mathbf p}}
\def\L{{\mathbf L}}
\def\DD{{\mathbb D}}
\def\EE{{\mathbb E}}
\def\MM{{\mathbb M}}
\def\JJ{{\mathbb J}}
\def\NN{{\mathbb N}}
\def\RR{{\mathbb R}}
\def\QQ{{\mathbb Q}}
\def\FF{{\mathbb F}}
\def\mm{{\mathfrak m}}
\def\nn{{\mathfrak n}}
\def\jj{{\mathfrak j}}
\def\aaa{{{{\mathfrak a}}}}
\def\bbb{{{{\mathfrak b}}}}
\def\ppp{{{{\mathfrak p}}}}
\def\qqq{{{{\mathfrak q}}}}
\def\PPP{{{{\mathfrak P}}}}
\def\BB{{\mathfrak B}}
\def\jj{{\mathfrak J}}
\def\LL{{\mathfrak L}}
\def\qq{{\mathfrak Q}}
\def\rr{{\mathfrak R}}
\def\cc{{\mathfrak S}}
\def\TT{{\mathcal{T}}}
\def\SS{{\mathcal S}}
\def\UU{{\mathcal U}}
\def\AA{{\mathcal A}}
\def\BB{{\mathcal B}}
\def\Primes{{\mathcal P}}
\def\genS{{\langle S \rangle}}
\def\genT{{\langle T \rangle}}
\def\bT{\mathsf{T}}
\def\bD{\mathsf{D}}
\def\bC{\mathsf{C}}
\def\VV{{\bf V}}
\def\ff{{\bf f}}
\def\uu{{\bf u}}
\def\aa{{\bf{a}}}
\def\bb{{\bf{b}}}
\def\zero{{\bf 0}}
\def\rad{\operatorname{rad}}
\def\End{\operatorname{End}}
\def\id{\operatorname{id}}
\def\mod{\operatorname{mod}}
\def\im{\operatorname{im}\,}
\def\ker{\operatorname{ker}}
\def\coker{\operatorname{coker}}
\def\ord{\operatorname{ord}}
\def\li{\operatorname{li}}
\def\Ei{\operatorname{Ei}}
\def\Ein{\operatorname{Ein}}
\def\Ri{\operatorname{Ri}}
\def\Rie{\operatorname{Rie}}
\def\degl{\operatorname{deglog}}

\newtheorem{theorem}{Theorem}[section]
\newtheorem*{theorem*}{Theorem}
\newtheorem{proposition}[theorem]{Proposition}
\newtheorem{corollary}[theorem]{Corollary}
\newtheorem{lemma}[theorem]{Lemma}
\newtheorem{unnumblemma}{Lemma}[section]

\theoremstyle{definition}
\newtheorem{example}[theorem]{Example}
\newtheorem{definition}[theorem]{Definition}
\newtheorem*{definition*}{Definition}
\newtheorem{remark}[theorem]{Remark}
\newtheorem{problem}[theorem]{Problem}
\newtheorem{conjecture}[theorem]{Conjecture}

 \newenvironment{map}[1]
   {$$#1:\begin{array}{rcl}}
   {\end{array}$$
   \\[-0.5\baselineskip]
 }

 \newenvironment{map*}
   {\[\begin{array}{rcl}}
   {\end{array}\]
   \\[-0.5\baselineskip]
 }

 \newenvironment{nmap*}
   {\begin{eqnarray}\begin{array}{rcl}}
   {\end{array}\end{eqnarray}
   \\[-0.5\baselineskip]
 }

 \newenvironment{nmap}[1]
   {\begin{eqnarray}#1:\begin{array}{rcl}}
   {\end{array}\end{eqnarray}
   \\[-0.5\baselineskip]
 }

\newcommand{\eq}{eq.\@\xspace}
\newcommand{\eqs}{eqs.\@\xspace}
\newcommand{\diagram}{diag.\@\xspace}

\numberwithin{equation}{section}


\title{Asymptotic expansions of the prime counting function \\ (final version)}

\author{Jesse Elliott \\  
California State University, Channel Islands \\
{\tt jesse.elliott@csuci.edu}}

\maketitle

\begin{abstract}
We provide several asymptotic expansions of the prime counting function $\pi(x)$ and related functions.   We define an {\it asymptotic continued fraction expansion} of a complex-valued function of a real or complex variable to be a possibly divergent continued fraction whose approximants provide an asymptotic expansion of the given function.   We show that,  for each positive integer $n$,  two well-known continued fraction expansions of the exponential integral function $E_n(z)$ correspondingly yield two asymptotic continued fraction expansions of $\pi(x)/x$.  We prove this by first establishing some general results about asymptotic  continued fraction expansions.   We show, for instance, that the ``best'' rational function approximations of a function possessing an asymptotic Jacobi  continued fraction expansion are precisely the approximants of the continued fraction, and as a corollary we determine all of the best rational function approximations of the function $\pi(e^x)/e^x$.  Finally, we generalize our results on $\pi(x)$ to any arithmetic semigroup satisfying Axiom A, and thus to any number field.  \\


\end{abstract}

\bigskip

\tableofcontents

\section{Introduction}

\subsection{The prime counting function $\pi(x)$}

This paper concerns the asymptotic behavior of the function $\pi: \RR_{>0} \longrightarrow \RR$ that for any $x > 0$ counts the number of primes less than or equal to $x$: $$\pi(x) =   \# \{p \leq x: p \mbox{ is prime}\}, \quad x > 0.$$  The function $\pi(x)$ is known as the {\bf prime counting function}. We call the related function $\PP: \RR_{> 0} \longrightarrow \RR$ defined by $$\PP(x) = \frac{\pi(x)}{x}, \quad x > 0,$$ the {\bf prime density function}.  The number $\PP(n)$ for any positive integer $n$ represents the probability that a randomly selected integer from $1$ to $n$ is prime.  

The celebrated {\bf prime number theorem}, proved independently by de la Vall\'ee Poussin \cite{val1} and Hadamard \cite{had}  in 1896,  states that
\begin{align*}
\pi(x) \sim \frac{x}{\log x} \ (x \to \infty),
\end{align*}
where $\log x$ is the natural logarithm.   The theorem can be expressed in the equivalent form $$\lim_{x \to \infty} x^{\PP(x)} = e,$$ which shows that the number $e$, like several other mathematical constants, encodes information about the distribution of the primes.   

The prime number theorem was first conjectured by Gauss in 1792 or 1793, according to Gauss' own recollection in his famous letter to Encke in 1849  \cite{gau}. The first actual published statement of something close to the conjecture was made by Legendre in 1798, which he refined further in 1808.  Following Legendre and Gauss, we let $A$ denote the unique function $[2,\infty) \longrightarrow \RR$ such that
$$\pi(x) = \frac{x}{\log x - A(x)}$$
for all $x \geq 2$, so that $A(x) = \log x - \frac{1}{\PP(x)}$.     Legendre's 1808 conjecture was that the limit $L := \lim_{x \to \infty} A(x)$ exists and is approximately equal to $1.08366$.   The limit $L$ is now often referred to as {\bf Legendre's constant}.   In 1848, Chebyshev further refined Legendre's conjecture by proving that if Legendre's constant exists then it must equal $1$ \cite{cheb}.  
Regarding the function $A$, Gauss made in his 1849 letter to Encke some very prescient remarks (English translation):
\begin{quote}
It appears that, with increasing $n$, the (average) value of $A$ decreases; however, I dare not conjecture whether the limit as $n$ approaches infinity is $1$ or a number different from 1.  I cannot say that there is any justification for expecting a very simple limiting value; on the other hand, the excess of $A$ over $1$ might well be a quantity of the order of $\frac{1}{\log n}$.
\end{quote}
Gauss' speculations  turned out to be correct, in that
\begin{align}\label{eq3b}
A(x) -1 \sim \frac{1}{\log x} \ (x \to \infty).
\end{align} 
A consequence of (\ref{eq3b}) is that Legendre's constant exists and is equal to $1$, which in turn implies the prime number theorem.

In 1838,  Dirichlet observed that $\pi(x)$ can be well approximated by the {\bf logarithmic integral function} 
$$\li(x) = \int_0^x \frac{dt}{\log t},$$
where the Cauchy principal value of the integral is assumed. Since it is straightforward to show using integration by parts that  $\li(x) \sim \frac{x}{\log x} \ (x \to \infty)$,
the prime number theorem is equivalent to
$$\pi(x) \sim \li(x) \ (x \to \infty).$$
The {\bf prime number theorem with error term}, proved  in 1899 by de la Vall\'ee Poussin \cite{val2}, states that there exists a constant $c > 0$ such that
$$\pi(x) = \li(x) + O\left(xe^{-c\sqrt{\log x}} \right) \ (x \to \infty).$$
Since $x^{t} = o( e^{c\sqrt{x}})  \ (x \to  \infty)$ for all $t>0$ and all $c > 0$, the prime number theorem with error term implies that
\begin{align}\label{PNTET}
\PP(x) = \frac{\li(x)}{x} + o\left(\frac{1}{(\log x)^{t}}\right) \ (x \to \infty),
\end{align}
for all $t > 0$, where also $\frac{\li(x)}{x}  = \frac{1}{x}\int_0^x \frac{dt}{\log t}$ is the principal average value of $\frac{1}{\log t}$ on the interval $[0,x]$.  This weaker form of the  prime number theorem with error term implies (\ref{eq3b}), and it is sufficient for many other applications, incuding all of those proved in this paper.   By contrast, more advanced research on the prime counting function is directed at {\it strengthening} the prime number theorem with error term, ideally to the conjecture $\PP(x) = \frac{\li(x)}{x} + O\left(\frac{\log x}{\sqrt{x}}\right) \ (x \to \infty)$, which,  famously,  is equivalent to the Riemann hypothesis \cite{koch}.

\subsection{Known asymptotic expansions of $\pi(x)$}


In this section we assume the reader is familiar with the notion of an {\it asymptotic expansion} of a function  \cite[Chapter 2]{copson}  \cite[Chapter 1]{erd} \cite[Chapter 1]{est}. Section 2.1 contains all of the background on asymptotic expansions needed for this paper. 

In 1848  \cite[p.\ 153]{cheb}, Chebyshev noted the asymptotic expansion
\begin{align}\label{asex}
\frac{\li(x)}{x} \simeq \sum_{k = 0}^\infty {\frac {k!}{(\log x)^{k+1}}} \ (x \to \infty).
\end{align}
This asymptotic expansion is now well known (see \cite[Section 10.3]{stop}, for example) and follows from  the fact that  $$\li(x) - \sum_{k = 0}^{n-1} \frac{k!x}{(\log x)^{k+1}} = \int_e^x \frac{n!dt} {(\log t)^{n+1}} +C_n \sim \frac{n!x}{(\log x)^{n+1}} \ (x \to \infty)$$ for all nonnegative integers $n$, where $C_n$ is a constant  and the given equality is proved by repeated integration by parts. 
 From (\ref{PNTET}) and  (\ref{asex}), we see that $\PP(x)$ has the same asymptotic expansion as $\frac{\li(x)}{x}$, namely,
\begin{align}\label{asex2}
\PP(x) \simeq \sum_{k = 0}^\infty {\frac {k!}{(\log x)^{k+1}}} \ (x \to \infty).
\end{align}
In fact, assuming (\ref{asex}), it follows easily  
 that (\ref{PNTET})  and (\ref{asex2}) are equivalent.  Thus, the asymptotic expansion (\ref{asex2}) carries essentially the same information as the weak version (\ref{PNTET}) of the prime number theorem with error term.  Note that the series $\sum_{k = 0}^\infty {\frac {k!}{(\log x)^{k+1}}}$ is divergent for all $x$, and the definition of asymptotic expansions equates (\ref{asex2}) with the statement
\begin{align*}
\PP(x) - \sum_{k = 0}^{n-1} {\frac {k!}{(\log x)^{k+1}}} \sim \frac{n!}{(\log x)^{n+1}} \ (x \to \infty), \quad \mbox{ for all } n \geq 1.
\end{align*}
which for $n = 2$ can be shown to yield  (\ref{eq3b}).



An interesting reformulation of the asymptotic expansion (\ref{asex2}) is  a result  proved by Panaitopol in 2000 \cite[Theorem]{pan}, namely,
\begin{align}\label{panthm}
\pi(x) = \frac{x}{\log x - 1-\sum_{n = 1}^{N} \frac{k_{n}}{(\log x)^{n}} + o \left( \frac{1}{(\log x)^{N}} \right)}  \ (x \to \infty), \quad  \text{ for all } N \geq 0,
\end{align}
where the sequence $\{k_n\}$ is determined by the recurrence relation appearing  in both \cite[Theorem]{pan} and  \cite[Theorem 5.2 (for $r = 2$)]{hall}  and is sequence A233824 of  the On-Line Encyclopedia of Integer Sequences (OEIS).  By those two theorems, $k_n$ for any $n \geq 0$ is the number of subgroups of index $n$ of the free group on two generators, and the sequence  $\{k_n\}$ has its first several terms given by  $0, 1, 3, 13, 71, 461, 3447, 29093, \ldots$.  It is also known that  $k_n = I_{n+1}$ for all $n$,  where $\{I_n\}$ is OEIS sequence A003319 and $I_n$ for any nonnegative integer $n$ is equal to the number of indecomposable permutations of $\{1,2,3\ldots, n\}$, where a permutation of $\{1,2,3,\ldots,n\}$ is said to be {\bf indecomposable} if it does not fix $\{1,2,3,\ldots,j\}$ for any  $1 \leq j < n$.

One can see that Panaitopol's result (\ref{panthm}) is equivalent to  (\ref{asex2}) as follows.  First, note that the asymptotic expansion  (\ref{asex2}), by the change of variable $x \longmapsto e^x$, has the equivalent ``linearized'' representation as
\begin{align}\label{pex}
\PP(e^x) \simeq \frac{1}{x} \sum_{k = 0}^\infty {\frac {k!}{x^{k}}} \ (x \to \infty).
\end{align}
Reciprocating the asymptotic expansion of $\PP(e^x)$ above yields the asymptotic expansion
\begin{align}\label{expans2}
A(e^x) = x- \frac{1}{\PP(e^x)}\simeq \sum_{k = 0}^\infty {\frac {N_k}{x^{k}}} \ (x \to \infty),
\end{align}
where $\{N_k\}$ is the sequence whose  generating function $\sum_{k = 0}^\infty N_k z^k$ is given by
$\frac{1}{z}-\frac{1}{z\sum_{k = 0}^\infty k!z^k}$.  But by \cite[(1) and (2)]{comt} the sequence $\{I_k\}$ has generating function $\sum_{k = 0}^\infty I_k z^k = 1-\frac{1}{\sum_{k = 0}^\infty k!z^k}$.  It follows that $N_k = I_{k+1}$ for all $k$.   Changing back to a linear scale, (\ref{expans2}) becomes
\begin{align}\label{Axsym}
A(x)  & \simeq \sum_{k = 0}^{\infty} \frac{I_{k+1}}{(\log x)^{k}} \ (x \to \infty),
\end{align}
which is a restatement  (\ref{panthm}) in terms of asymptotic expansions.   Clearly this argument is reversible, and therefore (\ref{asex2})--(\ref{Axsym}) are all equivalent.  These statements show in particular that the sequences $\{n!\}$, $\{k_n\}$, and $\{I_n\}$  encode information about the distribution of the primes.  We provide several other examples of this phenomenon in Section 4.2. 

   A graph of the function $A(e^x)$ and of the first five terms of the asymptotic expansion (\ref{expans2}) of $A(e^x)$ are provided in  Figure \ref{primeasym2b}.    Note, as a consequence of (\ref{Axsym}),  that $A(x) - 1 - \frac{1}{\log x} \sim \frac{3}{(\log x)^2} \ (x \to \infty),$ which in turn implies (\ref{eq3b}). 

\begin{figure}[ht!]
\centering
\includegraphics[width=110mm]{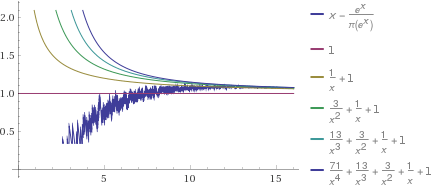}
\caption{Approximations of $A(e^x)$ \label{primeasym2b}}
\end{figure}


\subsection{Main results on $\pi(x)$}


The primary goal of this  paper is to derive various asymptotic expansions of the prime counting function, all of which either generalize or  are in some sense equivalent to  (\ref{asex2}).  The most novel of these are {\it asymptotic continued fraction expansions.}  We say that an {\bf asymptotic continued fraction expansion} of a real or complex function $f(x)$ is a possibly divergent continued fraction whose approximants  $w_n(x)$ provide an asymptotic expansion
$$f(x)- w_0(x) \simeq  \sum_{n = 1}^\infty(w_n(x)-w_{n-1}(x)) \ (x \to \infty)$$ of $f(x)-w_0(x)$ with respect to the sequence $\{w_{n}(x)-w_{n-1}(x)\}_{n = 1}^\infty$, which is assumed to be an asymptotic sequence, that is,   to satisfy $w_{n+1}(x)-w_{n}(x) = o(w_n(x)-w_{n-1}(x)) \ (x \to \infty)$ for all $n \geq 1$.  A more detailed definition is given in Section 2.2.   In Sections 2.3--2.5, we establish some general function-theoretic results about asymptotic  continued fraction expansions that have interest independently of the prime counting function.   The results in those sections, namely, Theorems \ref{wsimJ} and \ref{wsimG} and Corollary \ref{wsim}, are natural extensions of the theory of continued fractions to {\it asymptotic} continued fraction expansions that are worth making explict.    Theorem \ref{wsimJ}, for example, shows that the ``best'' rational function approximations of a function possessing an asymptotic Jacobi  continued fraction expansion are precisely the approximants of the continued fraction.   

Theorem \ref{maincontthm1} below provides two asymptotic continued fraction expansions of the function $\PP(x)$.  We  initially conjectured the theorem on the basis of the following observations.  Let $a \in \RR$.    Since
$$\frac{1}{\log x-a} = \sum_{k = 0}^\infty  {\frac {a^k}{(\log x)^{k+1}}} = \frac{1}{\log x}+ \frac{a}{(\log x)^2}+\frac{a^2}{(\log x)^3}+ O \left(\frac{1}{(\log x)^4} \right)\ (x \to \infty)$$
for all $x > e^a$,  the asymptotic expansion   (\ref{asex2}) of $\PP(x)$ implies that
\begin{align}\label{w1}
\PP(x) -\frac{1}{\log x-a} \sim \begin{cases} \displaystyle \cfrac{1-a}{(\log x)^{2}}   \ (x \to \infty)  & \text{if } a \neq 1 \\
 \cfrac{1}{(\log x)^{3}}   \ (x \to \infty)  & \text{if } a  = 1.
\end{cases}
\end{align}
 It follows that $a = 1$ is optimal in the approximation $\PP(x) \approx \frac{1}{\log x-a}$.
Since $A(x)-1 \sim \frac{1}{\log x} \ (x \to \infty)$,  it is natural, given the analysis above, to seek an $a \in \RR$ that is optimal in the approximation $A(x) - 1 \approx \frac{1}{\log x -a}$.  A similar analysis shows that Panaitopol's asymptotic expansion (\ref{Axsym}) of $A(x)$ yields the answer $a  = I_3= 3$:
\begin{align*}
A(x)-1 -\frac{1}{\log x-a} \sim \begin{cases}  \cfrac{3-a}{(\log x)^{2}}   \ (x \to \infty)  & \text{if } a \neq 3 \\
 \cfrac{4}{(\log x)^{3}}   \ (x \to \infty)  & \text{if } a  = 3.
\end{cases}
\end{align*}
(where $4 = I_4-3^2$).
We may then repeat the argument yet again by analyzing the unique function $B(x)$ such that $A(x) -1 = \frac{1}{\log x -B(x)}$.  This can be done by hand, but WolframAlpha (and Theorem \ref{wsimJ})  offers a shortcut: compute $\lim_{x \to \infty} ((B(x)-3)\log x)$ to find that the limit is $4$, so that $B(x) -3 \sim \frac{4}{\log x} \ (x \to \infty)$; then compute $\lim_{x \to \infty}\left(\log x - \frac{4}{B(x)-3}\right)$ to find that the limit is $5$, so that $b = 4$ and $a = 5$ are optimal in the approximation $B(x) - 3 \approx \frac{b}{\log x -a}$.  This can be repeated several more times.  These observations eventually led us to formulate the notion of an asymptotic continued fraction expansion  and  to conjecture that $\frac{1}{\log x - 1}$,  $\frac{1}{\log x - 1 \, -} \, \frac{1}{\log x - 3}$, and  $\frac{1}{\log x - 1 \, -} \, \frac{1}{\log x - 3 \, -} \, \frac{4}{\log x - 5 \, -}$ are the initial approximants of an asymptotic continued fraction expansion of $\PP(x)$. 

To prove the conjecture and derive the full expansion, we use Corollary \ref{wsim}, along with the asymptotic expansion (\ref{asex2}) and the result of Stieltjes \cite[No.\ 57]{stie} that the formal continued fraction $\frac{z} {1 \, -} \ \frac{z} {1 \, -}  \ \frac{z} {1 \, -} \ \frac{2z} {1 \, -}  \ \frac{2z} {1 \, -} \ \frac{3z} {1 \, -} \ \frac{3z} {1 \, -}  \, \cdots$
converges $(z)$-adically in $\CC[[z]]$ to  $\sum_{k = 0}^\infty k! z^{k+1}$,
to prove the following.

\begin{theorem}\label{maincontthm1}
One has the following asymptotic continued fraction expansions.
\begin{enumerate}
\item $\PP(x) \, \simeq \,  \cfrac{\frac{1}{\log x}}{1 \,-} \ \cfrac{\frac{1}{\log x}}{1 \,-}\  \cfrac{\frac{1}{\log x}}{1 \,-}\  \cfrac{\frac{2}{\log x}}{1 \,-}\  \cfrac{\frac{2}{\log x}}{1 \,-}\  \cfrac{\frac{3}{\log x}}{1 \,-} \  \cfrac{\frac{3}{\log x}}{1 \,-}\  \cfrac{\frac{4}{\log x}}{1 \,-}  \ \cfrac{\frac{4}{\log x}}{1 \,-}  \ \cdots \ (x \to \infty)$.
\item $\PP(x)  \,  \simeq \, \displaystyle \frac{1}{\log x -1\,-} \  \frac{1}{\log x  - 3 \,-}\  \frac{4}{\log x-5\,-}\  \frac{9}{\log x - 7 \,-} \ \frac{16}{\log x-9\,-}\   \cdots \  (x \to \infty)$.
\item $A(x) \,  \simeq \,   \cfrac{\frac{1}{\log x}}{1 \,-}\  \cfrac{\frac{1}{\log x}}{1 \,-}\  \cfrac{\frac{2}{\log x}}{1 \,-}\  \cfrac{\frac{2}{\log x}}{1 \,-}\  \cfrac{\frac{3}{\log x}}{1 \,-} \  \cfrac{\frac{3}{\log x}}{1 \,-}\  \cfrac{\frac{4}{\log x}}{1 \,-}  \ \cfrac{\frac{4}{\log x}}{1 \,-}  \ \cdots \ (x \to \infty)$.
\item $A(x)  \,  \simeq \, \displaystyle 1+  \frac{1}{\log x  - 3 \,-}\  \frac{4}{\log x-5\,-}\  \frac{9}{\log x - 7 \,-} \ \frac{16}{\log x-9\,-}\   \cdots \  (x \to \infty)$.
\end{enumerate}
\end{theorem}



To be more explicit, let $w_n(x)$ denote the $n$th approximant of the continued fraction in statement (2) of the theorem (so that $w_0(x) = 0$ and $w_1(x) = \frac{1}{\log x -1}$), which by \cite[J.3]{stie} coincides with the $2n$th approximant of the continued fraction in statement (1).  Statement (2) of the theorem is to be interpreted as  the asymptotic expansion
\begin{align}\label{req1}
\PP(x) \simeq   \sum_{n = 1}^\infty (w_{n}(x)-w_{n-1}(x)) \ (x \to \infty).
\end{align}
But, by \cite[(42.9)]{wall} (or by (\ref{detformula})), one has $w_{n+1}(x) - w_n(x) \sim \frac{(n!)^2}{(\log x)^{2n+1}} \ (x \to \infty)$ for all $n \geq 0$.  This implies, by the definition of asymptotic expansions,  that (\ref{req1}) is equivalent to
\begin{align}\label{man}
\PP(x) - w_n(x) \sim \frac{(n!)^2}{(\log x)^{2n+1}} \ (x \to \infty), \quad \text{ for all } n \geq 0.
\end{align}
Thus, statement (2) of Theorem \ref{maincontthm1} is essentially equivalent to (\ref{man}).  Note that (\ref{man}) for $n = 0$ is the prime number theorem, and for $n = 1$ it is equivalent to (\ref{w1}).

Theorem \ref{maincontthm1} is proved in Section 3.1.   It is also proved in that section that the approximants $w_n(x)$ above are precisely the ``best''  approximations of $\PP(x)$ that are rational functions of $\log x$, or, in other words, the functions $w_n(e^x)$ are precisely the ``best'' rational function approximations of $\PP(e^x)$, in the following sense: the functions $w_n(e^x)$ are precisely those rational functions $w \in \RR(x)$ such that $v = w$ for every rational function $v \in \RR(x)$ of degree at most $\deg w$ such that $\PP(e^x) - v(x) = O(\PP(e^x)-w(x)) \ (x \to \infty)$.      Figure \ref{primeasym3} provides a graph of the function $\PP(e^x)$ and the approximants $w_n(e^x)$ for $n = 1,2,3,4$.     Note that the denominator  of the rational function $w_n(e^x)$ is the {\it $n$th monic Laguerre polynomial} \cite{lag}.

\begin{figure}[ht!]
\centering
\includegraphics[width=110mm]{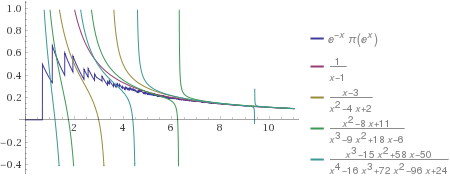}
\caption{The first four best nonconstant rational approximations of $\PP(e^x)$}\label{primeasym3}
\end{figure}

The notions of an asymptotic continued fraction expansion and a  best rational function approximation seem to be absent from the  existing literature,  and it was the discovery of  Theorem \ref{maincontthm1}   that motivated us to introduce them here.       Curiously, two more asymptotic continued fraction expansions of $\PP(x)$ and $A(x)$ follow  immediately from (\ref{asex2})  and the identity
\begin{align*}
\sum_{k = 0}^n k! z^{k+1} =  \frac{z}{1 \, -}  \ \frac{z} {1+z\, -} \ \frac{2z} {1+2z\, -} \   \frac{3z} {1+3z \, -} \ \frac{4z}{1+4z \, -} \  \cdots \ \frac{nz}{1+nz}, \quad n \geq 0.
\end{align*}
The identity can be verified easily by induction from the well-known recurrence relation  \cite[(1.4)]{wall}  for the numerator or denominator $C_n(z)$ of the  $n$th approximant  $u_n(z)$ of a continued fraction, which for the continued fraction
$$ \frac{z}{1 \, -}  \ \frac{z} {1+z\, -} \ \frac{2z} {1+2z\, -} \   \frac{3z} {1+3z \, -} \ \frac{4z}{1+4z \, -} \  \cdots \ = \  \cfrac{1}{\frac{1}{z} \, -}  \ \cfrac{\frac{1}{z}} {\frac{1}{z}+1\, -} \ \cfrac{\frac{2}{z}} {\frac{1}{z}+2\, -} \   \cfrac{\frac{3}{z}} {\frac{1}{z}+3 \, -} \ \cfrac{\frac{4}{z}}{\frac{1}{z}+4 \, -} \  \cdots$$
is the recurrence relation
$$C_{k+1}(z) = (1+kz)C_k(x) -(kz )C_{k-1}(z), \quad k \geq 1.$$   It follows that  $u_n(z) = \sum_{k = 0}^n \frac{k!}{(\log x)^{k+1}}$  for all $n$, where $z = \frac{1}{\log x}$.  From the definiton of an asymptotic continued fraction expansion, then, we obtain the four restatements of (\ref{asex2}) provided in the following proposition.

\begin{proposition}\label{gencornew}
One has the following asymptotic continued fraction expansions.
\begin{enumerate}
\item $\PP(x)\,  \simeq \,  \cfrac{\frac{1}{\log x}}{1 \, -}  \ \cfrac{\frac{1}{\log x}} {1+\frac{1}{\log x}\, -} \ \cfrac{\frac{2}{\log x}} {1+\frac{2 }{\log x}\, -} \   \cfrac{\frac{3}{\log x}} {1+\frac{3}{\log x} \, -} \ \cfrac{\frac{4}{\log x}}{1+\frac{4}{\log x} \, -}  \  \cdots \ (x \to \infty)$.
\item $\PP(x)  \,  \simeq \, \displaystyle  \frac{1}{\log x \, -}  \ \frac{1\log x} {\log x+1 \, -} \ \frac{2\log x} {\log x+2 \, -} \   \frac{3\log x} {\log x+3\, -} \ \frac{4\log x}{\log x+4 \, -}  \  \cdots \  (x \to \infty).$ 
\item $A(x) \, \simeq \,  \cfrac{\frac{1}{\log x}} {1+\frac{1}{\log x}\, -} \ \cfrac{\frac{2}{\log x}} {1+\frac{2 }{\log x}\, -} \   \cfrac{\frac{3}{\log x}} {1+\frac{3}{\log x} \, -} \ \cfrac{\frac{4}{\log x}}{1+\frac{4}{\log x} \, -}  \  \cdots \ (x \to \infty)$.
\item $A(x) \, \simeq \, \displaystyle  \frac{1\log x} {\log x+1 \, -} \ \frac{2\log x} {\log x+2 \, -} \   \frac{3\log x} {\log x+3\, -} \ \frac{4\log x}{\log x+4 \, -}  \  \cdots \  (x \to \infty)$.
\end{enumerate}
\end{proposition} 

 In Section 3.2, we prove the following generalization of Theorem \ref{maincontthm1}, along with an analogous generalization of  Proposition \ref{gencornew}.

\begin{theorem}\label{gentheorem}
For every nonnegative integer $n$ one has the asymptotic continued fraction expansions
$$p_n(x) \, \simeq \,  \cfrac{\frac{1}{\log x}}{1 \,-} \ \cfrac{\frac{1+n}{\log x}}{1 \,-}\  \cfrac{\frac{1}{\log x}}{1 \,-}\  \cfrac{\frac{2+n}{\log x}}{1 \,-}\  \cfrac{\frac{2}{\log x}}{1 \,-}\  \cfrac{\frac{3+n}{\log x}}{1 \,-} \  \cfrac{\frac{3}{\log x}}{1 \,-}\  \cfrac{\frac{4+n}{\log x}}{1 \,-}  \ \cfrac{\frac{4}{\log x}}{1 \,-}  \ \cdots \ (x \to \infty)$$
and
$$p_n(x)  \,  \simeq \, \frac{1}{\log x -1-n\,-} \  \frac{1(1+n)}{\log x - 3-n \,-}\  \frac{2(2+n)}{\log x-5-n \,-}\  \frac{3(3+n)}{\log x  - 7-n \,-}  \ \cdots \  (x \to \infty),$$
where  $p_n(x) = \frac{(\log x)^{n}}{n!} \left(\PP(x) - \sum_{k = 0}^{n-1} \frac{k!}{(\log x)^{k+1}}\right).$  Moreover, the same asymptotic continued fraction expansions hold for the function $l_n(x) = \frac{(\log x)^{n}}{n!} \left(\frac{\li(x)}{x} - \sum_{k = 0}^{n-1} \frac{k!}{(\log x)^{k+1}}\right).$
\end{theorem}

In Section  3.4, we provide an alternative proof of Theorem \ref{maincontthm1}  (resp., Theorem \ref{gentheorem}) from the measure-theoretic consequences of Stieltjes' theory of continued fractions  \cite{stie} applied to the  probability measure  $\gamma_0$ (resp., $\gamma_n$) on $[0,\infty)$ with density  function $e^{-t}$ (resp., $\frac{t^{n}}{n!}e^{-t}$), which is known as the {\bf exponential distribution with rate parameter $1$} (resp.,  {\bf gamma distribution with shape parameter $n+1$ and rate parameter $1$}).    The measure $\gamma_n$ is the unique Borel measure $\mu$ on $[0,\infty)$ whose Stieltjes transform $\int_{-\infty}^\infty \frac{d\mu(t)}{z-t}$ is equal to $ -e^{-z}E_{n+1}(-z)$ on $\CC\backslash[0,\infty)$, where $E_n(z) = z^{n-1}\int_{z}^\infty \frac{e^{-t}}{t^n} dt$ denotes the exponential integral function.   Our alternative proof  of Theorems \ref{maincontthm1} and Theorem \ref{gentheorem} explains the theorems' apparent connection to the well-known continued fraction expansions
\begin{align}
-e^{-z}E_{n+1}(-z) & =  \cfrac{\frac{1}{z}}{1 \,-} \  \cfrac{\frac{1+n}{z}}{1 \,-}\  \cfrac{\frac{1}{z}}{1 \,-}\  \cfrac{\frac{2+n}{z}}{1 \,-}\  \cfrac{\frac{2}{z}}{1 \,-}\  \cfrac{\frac{3+n}{z}}{1 \,-} \ \cfrac{\frac{3}{z}}{1 \,-}\  \cfrac{\frac{4+n}{z}}{1 \,-} \ \cfrac{\frac{4}{z}}{1 \,-} \ \cdots \label{Enexpansion} \\
  &=  \frac{1}{z -1-n\,-} \  \frac{1(1+n)}{z -3-n \,-}\  \frac{2(2+n)}{z-5-n \,-}\  \frac{3(3+n)}{z -  7-n \,-}  \ \cdots \label{Enexpansion2}
\end{align}
of $-e^{-z}E_{n+1}(-z)$ on $\CC\backslash [0,\infty)$, and to the fact that $l_n(e^x) =  \operatorname{Re}(-e^{-x}E_{n+1}(-x))$ for all $x \in (0, \infty)$.

\subsection{Outline of paper, and acknowledgments}  

Section 2  of this paper provides some background on asymptotic expansions (Section 2.1), introduces the notions of an {\it asymptotic continued fraction expansion} (Section 2.2) and a {\it best rational function approximation} (Section 2.3), and proves some general results about asymptotic continued fraction expansions  (Sections 2.3--2.5).  Sections 3.1 and 3.2 apply the results of Section 2 to the prime counting function, yielding, for example, Theorems \ref{maincontthm1} and \ref{gentheorem}.  Section 3.3 considers the meausure-theoretic aspects of Stieltes' theory of continued fractions  \cite{stie} and their consequences for asymptotic Stieltjes and Jacobi continued fraction expansions, and Section 3.4 provides an alternative proof of Theorems \ref{maincontthm1} and \ref{gentheorem} in this measure-theoretic context.  Section 4 relates the asymptotic expansion (\ref{asex2}) of the prime counting function to the harmonic numbers $H_n$ and to various combinatorially defined integer sequences (besides $\{n!\}$, $\{k_n\}$, and $\{I_n\}$).   Finally, Section 5 generalizes all of our results on $\pi(x)$ to any arithmetic semigroup satisfying a natural generalization of Axiom A  \cite{kno}, and thus in particular to any number field.  


All graphs provided in the paper were made using the free online version of WolframAlpha.  I am grateful to Kevin McGown, Roger Roybal, and  Hendrik W.\ Lenstra, Jr., for providing comments on various early drafts of this paper, and also to the anonymous referee(s), whose suggestions  improved the accuracy, organization, and clarity of the paper.

\section{Asymptotic continued fraction expansions}

\subsection{Asymptotic expansions}

In this section we provide some background on the theory of asymptotic expansions. 

Let $a$ be a limit point of a topological space $\XX$.  (The reader may assume throughout this section  that $\XX$ is a subspace of $\RR \cup \{\pm \infty\}$ or $\CC \cup \{\infty\}$; however, the definitions and results easily generalize to any topological space $\XX$, e.g., to $\XX$ a subspace of $({\mathbb C }\cup \{\infty\})^n$ for any positive integer $n$, and the more general presentation serves to simplify  the discussion.)  Let $f$ and $g$ be complex-valued functions whose domains are arbitrary sets.  We write $$f(x) = O(g(x)) \ (x \to a)_\XX$$ if for some $M > 0$ there exists a punctured neighborhood $U \subseteq \XX$ of $a$ such that $|f(x)| \leq M |g(x)|$ for all $x \in U$.   We write $$f(x) = o(g(x)) \ (x \to a)_\XX$$ if for every $M > 0$ there exists a punctured neighborhood $U \subseteq \XX$ of $a$ such that $|f(x)| \leq M |g(x)|$ for all $x \in U$.   Also, we write $$f(x) \sim g(x) \ (x \to a)_\XX$$ if $f(x) - g(x) = o(g(x))\ (x \to a)_\XX$.

\begin{remark}\label{arem} \
\begin{enumerate}
\item Both of the conditions $f(x) = o( g(x)) \ (x \to a)_\XX$ and $f(x)  \sim g(x) \ (x \to a)_\XX$ are stronger than the condition $f(x) = O(g(x)) \ (x \to a)_\XX$, and all three require that $f$ and $g$ be defined on a punctured neighborhood of $a$.  The relation $\sim$ is an equivalence relation, the relations $O$ and $o$ are transitive and invariant  under $\sim$-equivalence, on the class of all functions defined in a punctured neighborhood of $a$.   If $f(x) = O( g(x)) \ (x \to a)_\XX$ and $g(x) = o(h(x)) \ (x \to a)_\XX$, or if  $f(x) = o( g(x)) \ (x \to a)_\XX$ and $g(x) = O(h(x)) \ (x \to a)_\XX$,  then $f(x) = o( h(x)) \ (x \to a)_\XX$.  A function $f$ is zero in a punctured neighborhood of $a$ if and only if $f(x) \sim 0 \ (x \to a)_\XX$.
\item If $g$ is nonzero in a punctured neighborhood of $a$, then one has
\begin{align*}
f(x) = O(g(x)) \ (x \to a)_\XX  & \ \Longleftrightarrow \    \limsup_{x \to a} \left|\frac{f(x)}{g(x)}\right| < \infty, \\
f(x) = o(g(x)) \ (x \to a)_\XX  & \ \Longleftrightarrow \     \lim_{x \to a} \frac{f(x)}{g(x)} = 0, \\
f(x) \sim g(x) \ (x \to a)_\XX  & \ \Longleftrightarrow \   \lim_{x \to a} \frac{f(x)}{g(x)} = 1.
\end{align*}
\end{enumerate}
\end{remark}

An  {\bf asymptotic sequence  over $\XX$ at $a$} is a sequence $\{\varphi_n\}_{n = 1}^\infty$  of complex-valued functions $\varphi_n$ such that $\varphi_{n+1}(x) = o(\varphi_n(x)) \ (x \to a)_\XX$ for all $n \geq 1$.  Let $\{\varphi_n\}$ be an asymptotic sequence over $\XX$ at $a$, let $f$ be a complex-valued function, let $\{a_n\}$ be a sequence of complex numbers, and $N$ a positive integer.  The function $f$ is said to have an {\bf asymptotic expansion
\begin{align}\label{asympa}
f(x) \simeq \sum_{n=1}^{N} a_n \varphi_{n}(x) \ (x \to a)_\XX
\end{align}
of order $N$ over $\XX$  at $a$ with respect to $\{\varphi_n\}$} if
\begin{align}\label{asympb}
f(x) - \sum_{k=1}^{n} a_k \varphi_{k}(x) = o(\varphi_{n}(x)) \  (x \to a)_\XX
\end{align}
for all positive integers $n \leq N$ (or equivalently for $n = N$, given Remark \ref{asrem}(1) below).  The function $f$ is said to have an {\bf asymptotic expansion
\begin{align}\label{asympc}f(x) \simeq \sum_{n=1}^{\infty} a_n \varphi_{n}(x) \ (x \to a)_\XX
\end{align}
of order $\infty$ over $\XX$ at $a$ with respect to $\{\varphi_n\}$} if (\ref{asympb}) holds for all positive integers $n$.  In all of the definitions above, we replace ``$(x \to a)_\XX$'' with ``$(x \to a)$'' when $\XX$ is $\RR \cup \{\pm \infty\}$ or $\CC \cup \{\infty \}$.   Often we also replace ``$(x \to a)_\XX$'' with ``$(x \to a)_{\XX\backslash \{a\}}$'' when $a = \pm \infty$.

\begin{remark}\label{asrem} \
\begin{enumerate}
\item It is clear that, for a given positive integer $n$, condition (\ref{asympb}) implies that
\begin{align}\label{asymp3}
f(x) - \sum_{k=1}^{n-1} a_k \varphi_{k}(x) = O(\varphi_{n}(x)) \  (x \to a)_\XX,
\end{align}
which in turn implies that $f(x) - \sum_{k=1}^{n-1} a_k \varphi_{k}(x) = o(\varphi_{n-1}(x)) \  (x \to a)_\XX$
if $n \geq 2$.   Consequently, the asymptotic expansion (\ref{asympc}) holds if and only if either (\ref{asympb}) or (\ref{asymp3}) holds for infinitely many integers $n \geq 1$, if and only if (\ref{asymp3}) holds for all $n \geq 1$.
\item  For any positive integer $n$,  if $a_n \neq 0$, then (\ref{asympb}) is equivalent to
$$f(x) - \sum _{k=1}^{n-1}a_{k}\varphi _{k}(x) \sim a_n \varphi_n(x) \ (x \to a)_\XX.$$
\end{enumerate}
\end{remark}


\begin{example}   Important examples of asymptotic expansions over $\RR$ and $\CC$ follow from Taylor's theorem: if $f$ is real or complex function defined in a neighborhood of some number $a$, then one has an asymptotic expansion  $f(x) \simeq \sum_{n = 0}^N a_n (x-a)^n \ (x \to a)$
of $f$  at $a$ with respect to the asymptotic sequence $\{(x-a)^n\}$ if and only if $f$ is $N$ times differentiable at $a$, in which case  $a_n = \frac{f^{(n)}(a)}{n!}$ for all $n \leq N$.   This also applies  to $f$ over $\CC$ at  $\infty$ (resp., over $\RR$ at $\infty$, over $\RR$ at $-\infty$) by considering the function $f(1/x)$ with respect to the asymptotic sequence $\left\{ \frac{1}{x^k}\right\}$ at $0$ (resp., at $0^+$, at $0^-$).    As a consequence, two asymptotic expansions at $a$ with respect to  $\{(x-a)^n\}$ or at  $\infty$ with respect to $\left\{ \frac{1}{x^n}\right\}$ can be added, subtracted, multiplied, divided, and composed just like formal power series. 
This is utilized throughout Section 4.2 and justifies our generating function derivation of (\ref{expans2}).
\end{example}

References for the material above are \cite[Chapter 1]{erd} and \cite[Chapter 1]{est}.  However, most references on asymptotic expansions, including \cite{erd} and \cite{est},  require that there be a punctured neighborhood of $a$ on which all of the functions in a given asymptotic sequence are defined, or are even nonzero, while our definition only requires that each function individually is defined on a punctured neighborhood of $a$.   By Remark \ref{Laguerre}(1),  those stronger assumptions are too restrictive for the statement of Theorem \ref{maincontthm1} and for our definition  of an asymptotic continued fraction expansion given in the next section.

The following lemmas provide, respectively, a necessary and sufficient condition for two functions to have the same asymptotic expansion with respect to a given asymptotic sequence, and a natural condition under which two asymptotic sequences can be viewed as ``equivalent.''   

\begin{lemma}\label{asympprop}
Let $a$ be a limit point of a topological space $\XX$, let $\{\varphi_n\}$ be an asymptotic sequence over $\XX$ at $a$, let $N$ be a positive integer, and let $f$ and $g$ be complex-valued functions.  A given asymptotic expansion of $f$ of order $N$ over $\XX$ at $a$  with respect to $\{\varphi_n\}$  is also an asymptotic expansion of  $g$ of order $N$ over $\XX$ at $a$ with respect to $\{\varphi_n\}$ if and only if $f(x)-g(x) = o(\varphi_N(x)) \ (x \to a)_\XX.$
\end{lemma}

\begin{proof}
Suppose that $f(x) \simeq \sum_{n = 1}^N a_n \varphi_n(x) \ (x \to a)_\XX$ is an asymptotic expansion of $f$ of order $N$ over $\XX$  at $a$, or equivalently, that $f(x)-  \sum_{n = 1}^N a_n \varphi_n(x) = o(\varphi_N(x)) \ (x \to a)_\XX.$
If one has
$g(x)-  \sum_{n = 1}^N a_n \varphi_n(x) = o(\varphi_N(x)) \ (x \to a)_\XX,$
then subtracting we see that
$f(x)- g(x) = o(\varphi_N(x)) \ (x \to a)_\XX.$
The converse is also clear.  The lemma follows.
\end{proof}

\begin{lemma}\label{asympprop2a}
Let $a$ be a limit point of a topological space $\XX$, let $\{\varphi_n\}$ be an asymptotic sequence over $\XX$ at $a$, and let $\{\psi_n\}$ be  a sequence of complex-valued functions such that $\psi_n(x) -\varphi_n(x) = o(\varphi_N(x)) \ (x \to a)_\XX$ for all positive integers $n$ and $N$ with $n \leq N$.   Then $\{\psi_n\}$ is an asymptotic sequence over $\XX$ at $a$ with  $\varphi_n(x) -\psi_n(x) = o(\psi_N(x)) \ (x \to a)_\XX$ for all  positive integers $n$ and $N$ with $n \leq N$.  Moreover, any asymptotic expansion 
$f(x) \simeq \sum_{n = 1}^N a_n \varphi_n(x) \ (x \to a)_\XX$
of a complex-valued function $f$ over $\XX$ at $a$ with respect to $\{\varphi_n\}$ is equivalent to the asymptotic expansion $f(x) \simeq \sum_{n = 1}^N a_n \psi_n(x) \ (x \to a)_\XX$
of $f$ over $\XX$ at $a$ with respect to $\{\psi_n\}$.
\end{lemma}

\begin{proof}
Without loss of generality we may assume that $N$ is finite.  
For all $n$ one has $\psi_n(x) -\varphi_n(x) = o(\varphi_n(x)) \ (x \to a)_\XX$ and therefore $\psi_n(x) \sim \varphi_n(x) \ (x \to a)_\XX.$   It follows that $\varphi_n(x) -\psi_n(x) = o(\varphi_N(x)) = o(\psi_N(x)) \ (x \to a)_\XX$ for all $n \leq N$ and that $\{\psi_n\}$ is an asymptotic sequence over $\XX$ at $a$.  If the asymptotic expansion $f(x) \simeq \sum_{n = 1}^N a_n \varphi_n(x) \ (x \to a)_\XX$ holds, then
\begin{align*} f(x) - \sum_{n = 1}^N a_n \psi_n(x)  & = \left( f(x) - \sum_{n = 1}^N a_n \varphi_n(x) \right) +  \sum_{n = 1}^N a_n (\psi_n(x)-\varphi_n(x)) \\
  &=  o( \varphi_N(x))  \ (x \to a)_\XX \\ 
  &=  o( \psi_N(x))  \ (x \to a)_\XX,
\end{align*}
and therefore the asymptotic expansion $f(x) \simeq \sum_{n = 1}^N a_n \psi_n(x) \ (x \to a)_\XX$ also holds.  By symmetry, the reverse implication holds as well.  The lemma follows.
\end{proof}

\subsection{Asymptotic continued fraction expansions}

In this section we introduce the notion of an asymptotic continued fraction expansion.

Let $a$ be a limit point of a topological space $\XX$.  Let $f$, $b_0, b_1, b_2, \ldots$, and $a_1, a_2, \ldots$ be complex-valued functions, each of which individually has its domain containing a punctured neighborhood of $a$.   Consider the (formal) continued fraction
\begin{align*}
b_0(x) + \frac{a_1(x)}{b_1(x) \, +} \ \frac{a_2(x)}{b_2(x)\, +} \ \frac{a_3(x)}{b_3(x)\, +}  \ \cdots \ = \ b_0(x) +\cfrac{a_1(x)}{b_1(x)+\cfrac{a_2(x)}{b_2(x)+\cfrac{a_3(x)}{b_3(x)+\cdots}}}.
\end{align*}
Let 
\begin{align*}
w_n(x) =  b_0(x) + \frac{a_1(x)}{b_1(x)\, +} \ \frac{a_2(x)}{b_2(x)\, +} \ \cdots \ \frac{a_n(x)}{b_n(x)} \  = \ b_0(x) +\cfrac{a_1(x)}{b_1(x)+\cfrac{a_2(x)}{b_2(x)+\cdots +\cfrac{a_{n-1}(x)}{b_{n-1}(x)+\cfrac{a_{n}(x)}{b_{n}(x)}}}}
\end{align*} for all $n \geq 0$ denote the {\bf $n$th approximant} of the given continued fraction, where $w_0(x) = b_0(x)$ (and where the domain of $w_n(x)$ is the largest possible so that the expression above is defined).  We write 
\begin{align}\label{contlega} 
f(x) \, \simeq \, b_0(x) + \frac{a_1(x)}{b_1(x) \,+} \ \frac{a_2(x)}{b_2(x)\,+} \ \frac{a_3(x)}{b_3(x)\,+} \ \cdots  \ (x \to a)_\XX
\end{align}
if  $\{w_n(x)-w_{n-1}(x)\}_{n = 1}^\infty$ is an asymptotic sequence over $\XX$ at $a$  and 
$$f(x) -w_0(x) \simeq \sum_{n = 1}^\infty (w_n(x)-w_{n-1}(x))  \ (x \to a)_\XX$$
is an  asymptotic expansion of $f(x)$ over $\XX$ with respect to $\{w_n(x)-w_{n-1}(x)\}$.  In that case we say that (\ref{contlega}) is an {\bf asymptotic continued fraction expansion of $f$ over $\XX$ at $a$.}   Thus, (\ref{contlega}) is an asymptotic continued fraction expansion of $f$ over $\XX$ at $a$ if and only if 
\begin{align}\label{contas1}
w_{n+1}(x)-w_{n}(x) = o(w_n(x)-w_{n-1}(x)) \ (x \to a)_\XX
\end{align}
and
\begin{align}\label{contas2}
f(x) - w_{n}(x) = o( w_{n}(x)-w_{n-1}(x)) \ (x \to a)_\XX
\end{align}
for all $n \geq 1$. 


\begin{remark}\label{crem} \
\begin{enumerate}
\item By Remark \ref{asrem}, condition (\ref{contas2})  can be replaced with
\begin{align}\label{contas3}
f(x) - w_{n}(x)  = O( w_{n+1}(x)-w_{n}(x)) \ (x \to a)_\XX
\end{align}
and with
\begin{align}\label{contas4}
f(x) - w_{n}(x) \sim w_{n+1}(x)-w_{n}(x) \ (x \to a)_\XX.
\end{align}
More precisely, the asymptotic continued fraction  expansion (\ref{contlega}) holds if and only if (\ref{contas1}) holds for all $n \geq 1$ and any one of the three conditions (\ref{contas2}), (\ref{contas3}), (\ref{contas4}) holds for infinitely many $n \geq 1$, in which case all three of the conditions hold for all  $n \geq 1$.
\item $\{w_{n}(x)-w_{n-1}(x)\}$ is an asymptotic sequence over $\XX$ at $a$ if and only if $w_{l}(x)-w_n(x) \sim w_m(x)-w_n(x) \ (x \to a)_\XX$ for all $n,m,l  \geq 0$ with $m > n$ and $l > n$, or equivalently with $m = n+1$ and $l = n+2$.
\item The asymptotic continued fraction expansion (\ref{contlega}) holds if and only if 
$f(x) - w_n(x) \sim w_{m}(x)-w_n(x)  \ (x \to a)_\XX$
 for all $n,m  \geq 0$ with $m > n$, or equivalently with $n < m \leq n+2$.
\item  If the asymptotic continued fraction expansion (\ref{contlega}) holds, then
$$f(x) \sim b_0(x) \ (x \to a)_\XX \ \Longleftrightarrow\  w_1(x) \sim b_0(x) \ (x \to a)_\XX \ \Longleftrightarrow\ f(x) \sim w_n(x) \ (x \to a)_\XX \text{ for all } n.$$
\item If $f(x)-b_0(x)$ and $a_1(x)$ are nonzero in a neighborhood of $a$, then the asymptotic continued fraction  expansion (\ref{contlega}) holds if and only if $\frac{a_1(x)}{f(x) -b_0(x)}  \sim  b_1(x) \ (x \to \infty)_\XX$ and one has the asymptotic  continued fraction expansion
\begin{align*}
\frac{a_1(x)}{f(x) -b_0(x)} \, \simeq \,  b_1(x)+ \ \frac{a_2(x)}{b_2(x)\,+} \ \frac{a_3(x)}{b_3(x)\,+} \ \cdots  \ (x \to a)_\XX.
\end{align*}

\end{enumerate}
\end{remark}


For ease of notation one makes the identification
\begin{align*}
b_0(x) \pm \frac{a_1(x)}{b_1(x) \, -} \ \frac{a_2(x)}{b_2(x)\, -} \ \frac{a_3(x)}{b_3(x)\, -}  \ \cdots \ = \
b_0(x)+ \frac{\pm a_1(x)}{b_1(x) \, +} \ \frac{-a_2(x)}{b_2(x)\, +} \ \frac{-a_3(x)}{b_3(x)\, +}  \ \cdots,
\end{align*}
and one identifies two formal continued fractions if they possess identical approximants.

The field $\CC((z)) = \CC[[z]][1/z]$ of formal Laurent series in $z$ is a complete normed field in the topology induced by the $(z)$-adic topology on the DVR $\CC[[z]]$.  We say that a formal continued fraction
$$G(z) = b_0(z)+  \frac{a_1(z)}{b_1(z) \, +} \  \frac{a_2(z)}{b_2(z) \,+} \  \frac{a_3(z)}{b_3(z) \,+} \ \cdots$$
with $a_n(z),b_n(z) \in \CC((z))$ for all $n$ {\bf converges formally} to  $F(z) \in \CC((z))$ if $F(z) = \lim_{n \to \infty} w_n(z)$ in the normed field $\CC((z))$, or equivalently if $F(z) = w_0(z)+ \sum_{n = 1}^\infty( w_n(z)-w_{n-1}(z))$ in the normed field $\CC((z))$, where $w_n(z)$ is the $n$th approximant of $G(z)$.  For further details on formal convergence, see \cite[Chapter V]{lor}.  Note that formal convergence is a special case of the more general notion of the convergence of a continued fraction over a normed field---and all of the other notions defined in Sections 2.1 and 2.2 can likewise be so generalized.

\subsection{Jacobi continued fractions and best rational approximations}

In this section  we prove some general results about asymptotic Jacobi continued fraction expansions.

The {\bf degree} $\deg w$ of a rational function $w \in K(X)$ over a field $K$ is equal to the maximum of the degree of the numerator and  the degree of denominator of $w$ when $w$ is written as a quotient of two relatively prime polynomials in $K[X]$.  Let $f$ be a complex function, and let $\XX$ be an unbounded subset of $\CC$.  We say that a rational  function $w \in \CC(z)$ is a {\bf best rational approximation of $f(z)$ over $\XX$ (at $\infty$)}  if $w$ is the unique rational function $v \in \CC(z)$  of degree at most $\deg w$  such that $f(z)-v(z) = O(f(z)-w(z)) \ (z \to \infty)_\XX$.    


 A {\bf Jacobi continued fraction}, or {\bf J-fraction}, is a continued fraction of the form
\begin{align*}
 \frac{a_1}{z+b_1  \,-} \  \frac{a_2}{z+b_2 \,-} \  \frac{a_3}{z+b_3 \,-}\ \cdots,
\end{align*}
where $a_n, b_n \in \CC$ and $a_n \neq 0$ for all $n$. The following theorem characterizes asymptotic Jacobi continued fraction expansions and shows that, if a given complex function has an asymptotic Jacobi continued fraction expansion, then the best rational approximations of the function are precisely the approximants of the continued fraction. 

\begin{theorem}\label{wsimJ}
Let $\{a_n\}$ and $\{b_n\}$ be sequences of complex numbers with $a_n \neq 0$ for all $n$, and for all nonnegative integers $n$ let $w_n(z)$ denote the $n$th approximant of the Jacobi continued fraction
\begin{align*}
\frac{a_1}{z+b_1 \,-} \  \frac{a_2}{z+b_2 \,-} \  \frac{a_3}{z+b_3 \,-} \ \cdots.
\end{align*}
One has the following.
\begin{enumerate}
\item $w_n(z)$ is a rational function of degree $n$  with $w_n(z) \sim \frac{a_1}{z}  \ (z \to \infty)$ and $$w_n(z)-w_{n-1}(z) \sim \frac{a_1 a_2 \cdots a_n}{z^{2n-1}}  \ (z \to \infty)$$ for all $n \geq 1$.  In particular, $\{w_n(z)-w_{n-1}(z)\}_{n = 1}^\infty$ is an asymptotic sequence at $\infty$.  Moreover, the given continued fraction converges formally to a series  $\sum_{k = 1}^\infty \frac{c_k}{z^k}$ in $(1/z)\CC[[1/z]]$.
\item  Let $f(z)$ be a complex function, and let $\XX$ be an unbounded subset of $\CC$.  The following conditions are equivalent.
\begin{enumerate}
\item $f(z)$ has the asymptotic expansion  $f(z) \simeq \sum_{k= 1}^\infty \frac{c_k}{z^k}\ (z \to \infty)_{\XX}$.
\item $f(z)$  has the asymptotic continued fraction expansion
\begin{align*}
f(z) \, \simeq \, \frac{a_1}{z+b_1 \,-} \  \frac{a_2}{z+b_2 \,-} \  \frac{a_3}{z+b_3 \,-} \ \cdots  \ (z \to \infty)_\XX.
\end{align*}
\item $f(z) - w_n(z) \sim  \frac{a_1 a_2 \cdots a_{n+1}}{z^{2n+1}} \ (z \to \infty)_\XX$ for all  nonnegative integers $n$.
\item $f(z) - w_n(z) = O\left(\frac{1}{z^{2n+1}} \right) \ (z \to \infty)_\XX$ for infinitely many (or all) nonnegative integers $n$.
\item $w_n(z)$ for every nonnegative integer $n$ is the unique rational function  $w(z) \in \CC(z)$ of degree at most $n$ such that $f(z) - w(z) = O\left(\frac{1}{z^{2n+1}} \right) \ (z \to \infty)_\XX$.  
\item For every positive integer $n$ one has $f_n(z) \sim \frac{a_{n+1}}{z} \ (z \to \infty)_\XX$, where $f_n(z)$ is the function defined by the equation
\begin{align*}
f(z) =  \frac{a_1}{z+b_1 \,-} \  \frac{a_2}{z+b_2 \,-} \  \frac{a_3}{z+b_3 \,-} \ \cdots  \  \frac{a_{n-1}}{z+b_{n-1} \,-} \  \frac{a_n}{z+b_n - f_n(z)},
\end{align*}
or equivalently by the recurrence relation $f_{n+1}(z) = z+b_{n+1}-\frac{a_{n+1}}{f_n(z)}$, where $f_0(z) = f(z)$.
\item For every positive integer $n$, and for some (or all) $d_n \in \CC$, the function 
\begin{align*}
f(z) -  \frac{a_1}{z+b_1 \,-} \  \frac{a_2}{z+b_2 \,-} \  \frac{a_3}{z+b_3 \,-} \ \cdots \   \frac{a_{n-1}}{z+b_{n-1}\, -} \  \frac{a_n}{z +b_n+d_n}
\end{align*}
is asymptotic over $\XX$ at $\infty$ to $\frac{a_1 a_2 \cdots a_n}{z^{2n}}d_n$ if $d_n \neq 0$, or to $\frac{a_1a_2\cdots a_{n+1}}{z^{2n+1}}$ if $d_n = 0$.
\end{enumerate}
\item If the equivalent conditions of statement (2) hold, then the best rational approximations of $f(z)$ over $\XX$ are precisely the approximants $w_n(z)$ for $n \geq 0$.
\end{enumerate}
\end{theorem}

\begin{proof} 
The (appropriately normalized) numerator $A_n = A_n(z)$ and denominator $B_n = B_n(z)$ of $w_n(z) = \frac{ A_n(z)}{B_n(z)}$ are uniquely determined by the well-known recurrence relations \cite[(1.4)]{wall} (or \cite[(1.3.1a) and (1.3.1b)]{cuyt}) and therefore by induction are polynomials of degree $n-1$ and $n$, respectively, with $B_n$ monic and $A_n$ having leading coefficient $a_1$.    Moreover, by \cite[(42.9)]{wall} (which follows from the well-known determinant formula \cite[(1.5)]{wall} \cite[(1.3.4)]{cuyt}), one has 
\begin{align}\label{detformula}
w_n(z)-w_{n-1}(z) = \frac{a_1 a_2 \cdots a_n}{B_n(z)B_{n-1}(z)},
\end{align}
and therefore $w_n(z)-w_{n-1}(z) \sim \frac{a_1 a_2 \cdots a_n}{z^{2n-1}}  \ (z \to \infty)$.
It also follows from (\ref{detformula}) that
\begin{align}\label{Jfrac}
w_{n}(z)-w_{n-1}(z)  = \frac{a_1 a_2 \cdots a_n}{z^{2n-1}} F_n(z)
\end{align}
for some $F_n(z) \in \CC[[1/z]]$ with constant term $1$.  From  this it follows that the $w_n(z)$ converge $(1/z)$-adically to a series 
$\sum_{k = 1}^\infty \frac{c_k}{z^k} = \sum_{n = 1}^\infty (w_n(z)-w_{n-1}(z))$ in $(1/z)\CC[[1/z]]$.  Statement (1) follows.

By statement (1) and Remark \ref{crem}(1), statements (2)(b)--(d) are equivalent.  We verify the equivalence of (2)(a) and (2)(d) as follows. 
By (\ref{Jfrac}), one has
$$\sum_{k= 1}^\infty \frac{c_k}{z^k} - w_n(z)  = \sum_{k = n+1}^\infty (w_k(z)-w_{k-1}(z)) = \frac{G_n(z)}{z^{2n+1}}$$
for some $G_n(z) \in \CC[[1/z]]$ (with constant term $a_1 a_2 \cdots a_{n+1}$).  It follows that
$$\sum_{k = 1}^{2n} \frac{c_k}{z^k} - w_n(z)  = \frac{H_n(z)}{z^{2n+1}}$$
for some $H_n(z) \in  \CC[[1/z]]$.  Since then $H_n(z)$ is also a rational function of $1/z$, it is analytic at $\infty$,  and thus
$$\sum_{k = 1}^{2n} \frac{c_k}{z^k} - w_n(z)  = O\left( \frac{1}{z^{2n+1}} \right) \ (z \to \infty).$$  Therefore, one has
$$f(z) - \sum_{k = 1}^{2n} \frac{c_k}{z^k}  = O\left( \frac{1}{z^{2n+1}} \right) \ (z \to \infty)_{
\XX}, \quad \text{ for all } n \geq 0$$
if and only if 
$$f(z) - w_n(z)  = O\left( \frac{1}{z^{2n+1}} \right) \ (z \to \infty)_{
\XX}, \quad \text{ for all } n \geq 0.$$
This proves that statements (2)(a)--(d) are equivalent.  

Alternatively, the equivalence of statements (2)(a)--(d) follows from Theorem \ref{wsimG}, which is stated and proved in Section 2.5.  From that theorem, it also follows that  (2)(f) and (2)(g) are both equivalent to (2)(a)--(d).  Rather than repeat the proof here, we defer the proof of those two equivalences until Section 2.5.  To finish the proof of statement (2), then, we show that (2)(d) and (2)(e) are equivalent.  Clearly (2)(e) implies (2)(d).  To prove the converse,  suppose that condition (2)(d) holds (for all $n$).  Let $v(z)$ be any rational function of degree at most $n$ such that $f(z) - v(z) = O\left(\frac{1}{z^{2n+1}} \right) \ (z \to \infty)_\XX$.  Then by (2)(c) one also has $$w_n(z) - v(z) = (f(z)-v(z)) - (f(z)-w_n(z)) =  O\left(\frac{1}{z^{2n+1}} \right) \ (z \to \infty)_\XX.$$  But since $w_n(z)- v(z)$ is a rational function of degree at most $\deg w_n(z)+ \deg v(z) \leq 2n < 2n+1$, it follows that $w_n(z) - v(z) = 0$.  Thus (2)(d) and (2)(e) are equivalent.

A similar argument shows that condition (2)(d) implies that $w_n(z)$ is a best rational approximation of $f(z)$ over $\XX$:  if $v(z)$ is any rational function of degree at most $n$ such that $f(z) - v(z) = O(f(z)-w_n(z)) \ (z \to \infty)_\XX$, then (2)(d) implies that $f(z) - v(z) =  O\left(\frac{1}{z^{2n+1}} \right) \ (z \to \infty)_\XX$, whence again we conclude that $v(z) = w_n(z)$.   To complete the proof of statement (3),  suppose that conditions (2)(a)--(g) hold, and let $v(z)$ be any best rational approximation of $f(z)$ over $\XX$ of degree $n = \deg w_n(z)$.  It remains only to show that $v(z) = w_n(z)$.  Since  both $ w_n(z)-v(z)$  and $\frac{1}{z^{2n+1}}$ are rational functions, one has either 
\begin{align}\label{hyp}
w_n(z)-v(z) =  O\left(\frac{1}{z^{2n+1}} \right)  \ (z \to \infty)
\end{align}
or
\begin{align}\label{hyp2}
\frac{1}{z^{2n+1}} =  o(w_n(z)-v(z))  \ (z \to \infty).
\end{align}
We show that (\ref{hyp2}) is impossible.  Suppose to obtain a contradiction that (\ref{hyp2}) holds, so, in particular, $w_n(z) \neq v(z)$.  By statement (2)(c) and (\ref{hyp2})  one has
$$f(z)-w_n(z) \sim \frac{a_1 a_2 \cdots a_{n+1}}{z^{2n+1}} =  o(w_n(z)-v(z)) \ (z \to \infty)_\XX$$
and therefore
$$f(z)-v(z)  \sim w_n(z)-v(z) \ (z \to \infty)_\XX,$$
so that also
$$f(z)-w_n(z) =  o(f(z)-v(z)) \ (z \to \infty)_\XX.$$
But then $w_n(z) = v(z)$ since $v(z)$ is a best rational approximation of $f(z)$, which is our desired contradiction.    Therefore (\ref{hyp}) must hold, so by (2)(d) one has
$$f(z)-v(z) = (f(z)-w_n(z))+ (w_n(z)-v(z)) =  O\left(\frac{1}{z^{2n+1}} \right) \ (z \to \infty)_\XX.$$  
Therefore,  by (2)(e), one has $v(z) = w_n(z)$, as desired.
\end{proof}

\subsection{Stieltjes continued fractions}

Continued fractions of the form given in the following result, which is more or less a corollary of Theorem \ref{wsimJ}, were introduced by Stieltjes in \cite[J.3]{stie} and are called {\bf Stieltjes continued fractions}, or {\bf S-fractions}.

\begin{corollary}\label{wsim}
Let $\{a_n\}$ be a sequence of nonzero complex numbers, and for all nonnegative integers $n$ let  $w_n(z)$ denote the $n$th approximant of the Stieltjes continued fraction
\begin{align*}
 \cfrac{\frac{a_1}{z}} {1 \, -} \ \cfrac{\frac{a_2}{z}} {1 \, -}  \ \cfrac{\frac{a_3}{z}} {1 \, -} \ \cfrac{\frac{a_4}{z}} {1 \, -} \  \cdots \ = \  \frac{a_1} {z \, -} \ \frac{a_2} {1 \, -}  \ \frac{a_3} {z \, -} \ \frac{a_4} {1 \, -} \  \cdots.
\end{align*}
One has the following.
\begin{enumerate}
\item $w_n(z)$ is a rational function of degree ${\left\lfloor \frac{n+1}{2}\right\rfloor}$  with $w_n(z) \sim \frac{a_1}{z}  \ (z \to \infty)$ and $$w_n(z)-w_{n-1}(z) \sim \frac{a_1 a_2 \cdots a_n}{z^{n}}  \ (z \to \infty)$$ for all $n \geq 1$.   In particular, $\{w_n(z)-w_{n-1}(z)\}_{n = 1}^\infty$ is an asymptotic sequence at $\infty$.  Moreover, the $2n$th approximant $w_{2n}(z)$ of the Stieltjes continued fraction above coincides with the $n$th approximant of the Jacobi continued fraction
\begin{align*}
\frac{a_1}{z-a_2 \,-} \  \frac{a_2 a_3}{z-a_3-a_4 \,-} \  \frac{a_4a_5}{z-a_5-a_6 \,-}  \ \cdots,
\end{align*}
and both continued fractions converge formally to a series  $\sum_{k = 1}^\infty \frac{c_k}{z^k}$ in $(1/z)\CC[[1/z]]$.
 \item Let $f(z)$ be a complex function, and let $\XX$ be an unbounded subset of $\CC$.    The following conditions are equivalent.
\begin{enumerate}
\item $f(z)$ has the asymptotic expansion  $f(z) \simeq \sum_{k= 1}^\infty \frac{c_k}{z^k}\ (z \to \infty)_{\XX}$.
\item $f(z)$  has the asymptotic continued fraction expansion
\begin{align*}
f(z) \, \simeq \,  \cfrac{\frac{a_1}{z}} {1 \, -} \ \cfrac{\frac{a_2}{z}} {1 \, -}  \ \cfrac{\frac{a_3}{z}} {1 \, -} \ \cfrac{\frac{a_4}{z}} {1 \, -} \  \cdots  \ (z \to \infty)_\XX.
\end{align*}
\item $f(z)$  has the asymptotic continued fraction expansion
\begin{align*}
f(z) \, \simeq \, \frac{a_1}{z-a_2 \,-} \  \frac{a_2 a_3}{z-a_3-a_4 \,-} \  \frac{a_4a_5}{z-a_5-a_6 \,-} \  \cdots  \ (z \to \infty)_\XX.
\end{align*}
\item $f(z) - w_n(z) \sim  \frac{a_1 a_2 \cdots a_{n+1}}{z^{n+1}} \ (z \to \infty)_\XX$ for all  nonnegative integers $n$. 
\item $f(z) - w_n(z) = O\left(\frac{1}{z^{n+1}} \right) \ (z \to \infty)_\XX$ for  infinitely many (or all)   nonnegative integers $n$.
\item $w_{2n}(z)$ for every nonnegative integer $n$ is the unique rational function  $w(z) \in \CC(z)$ of degree at most $n$ such that $f(z) - w(z) = O\left(\frac{1}{z^{2n+1}} \right) \ (z \to \infty)_\XX$.  
\end{enumerate}
\item If the equivalent conditions of statement (2) hold, then the best rational approximations of $f(z)$ over $\XX$ are precisely the even-indexed approximants $w_{2n}(z)$. 
\end{enumerate}
\end{corollary}

\begin{proof}
Statement (1) follows from the well-known formulas \cite[(1.4)--(1.5)]{wall} for continued fraction approximants and from the transformation  \cite[(28.2)--(28.4)]{wall}  from Stieltjes  to Jacobi continued fractions, and statements (2) and (3) follow immediately from statement (1) and Theorem \ref{wsimJ}.
\end{proof}

\begin{remark}
Let us say that a rational  function $w \in \CC(z)$ is a {\bf good rational approximation of $f(z)$ over $\XX$ (at $\infty$)} if $\deg v \geq \deg w$ for any any $v \in \CC(z)$ such that $f(z)-v(z) = O(f(z)-w(z)) \ (z \to \infty)_\XX$.  Clearly any best rational approximation is a good rational approximation.
If $f(z)$ satisfies conditions (2)(a)--(f) of Corollary \ref{wsim}, then all of the approximants $w_n(z)$, not just the even-indexed ones, are good rational approximations of $f(z)$ over $\XX$.  (Indeed, if $f(z)-v(z) = O(f(z)-w_n(z)) \ (z \to \infty)_\XX$, then $w_n(z)-v(z) = O\left(\frac{1}{z^{n+1}}\right) \ (z \to \infty)_\XX$; but then $\deg v < \deg w_n = {\left\lfloor \frac{n+1}{2}\right\rfloor}$ implies $\deg w_n+ \deg v < n+1$, which leads to a contradiction, whence $\deg v \geq \deg w_n$.)
\end{remark}

See Sections 3.1 and 3.2 for a proof of Theorems \ref{maincontthm1} and \ref{gentheorem} from Corollary \ref{wsim}.
\subsection{Continued fractions with polynomial terms}

The following  theorem shows that any asymptotic continued fraction expansion at $0$ with respect to a continued fraction that has nonzero terms $a_n(z)$ in $z\CC[z]$ and $b_n(z)$ in $\CC[z]\backslash z\CC[z]$ is equivalent to an asymptotic expansion with respect to the asymptotic sequence $\{z^n\}$  at $0$.  

\begin{theorem}\label{wsimG}
Let $b_0(z) \in \CC[z]$, let $a_n(z), b_n(z) \in \CC[z]$ with $a_n(0) = 0$ to multiplicity $m_n \in \ZZ_{> 0}$ and $b_n(0) =  \beta_n \neq 0$ for all $n \geq 1$, and let  $w_n(z)$ denote the $n$th approximant of the continued fraction
\begin{align*}
 b_0(z)+  \frac{a_1(z)}{b_1(z) \, -} \  \frac{a_2(z)}{b_2(z) \,-} \  \frac{a_3(z)}{b_3(z) \,-} \ \cdots
\end{align*}
One has the following.
\begin{enumerate}
\item $w_n(z)$ is a rational function in $\CC(z)$ with $w_n(z)-b_0(z) \sim \frac{a_1(z)}{\beta_1}  \ (z \to 0)$ and $$w_n(z)-w_{n-1}(z) \sim \frac{a_1(z) a_2(z) \cdots a_n(z)}{(\beta_1\beta_2\cdots \beta_{n-1})^2\beta_n}  \ (z \to 0)$$ for all $n \geq 1$.  Moreover, $\{w_n(z)-w_{n-1}(z)\}_{n = 1}^\infty$ is an asymptotic sequence at $0$, and the given continued fraction converges $(z)$-adically to a series $\sum_{n = 0}^\infty c_n z^n$ in $\CC[[z]]$.  
 \item Let $f(z)$ be a complex function, and let $\XX$ be a subset of $\CC$ having $0$ has a limit point.  The following conditions are equivalent.
\begin{enumerate}
\item $f(z)$ has the asymptotic expansion $f(z) \, \simeq \, \sum_{n = 0}^\infty c_n z^n \ (x \to 0)_{\XX}$.
\item $f(z)$  has the asymptotic continued fraction expansion
\begin{align*}
f(z) \, \simeq \,   b_0(z)+  \frac{a_1(z)}{b_1(z) \, -} \  \frac{a_2(z)}{b_2(z) \,-} \  \frac{a_3(z)}{b_3(z) \,-} \ \cdots \   \cdots \ (z \to 0)_\XX.
\end{align*}
\item $f(z) - w_n(z) \sim  \frac{a_1(z) a_2(z) \cdots a_{n+1}(z)}{(\beta_1\beta_2\cdots \beta_n)^2\beta_{n+1}} \ (z \to 0)_\XX$ for all  nonnegative integers $n$. 
\item $f(z) - w_n(z) = O\left( z^{m_1+m_2 + \cdots+ m_{n+1}} \right) \ (z \to 0)_\XX$ for  infinitely many (or all)   nonnegative integers $n$.
\item  For every positive integer $n$ one has $f_n(z) \sim \frac{a_{n+1}(z)}{\beta_{n+1}} \ (z \to 0)_\XX$, where $f_n(z)$ is the function defined by the equation
\begin{align*}
f(z) =   b_0(z)+  \frac{a_1(z)}{b_1(z) \, -} \  \frac{a_2(z)}{b_2(z) \,-} \  \frac{a_3(z)}{b_3(z) \,-} \ \cdots \   \frac{a_{n}(z)}{b_{n}(z) \,-} \   \frac{f_n(z)}{1},
\end{align*}
or equivalently by the recurrence relation $f_{n+1}(z) = b_{n+1}(z)-\frac{a_{n+1}(z)}{f_n(z)}$, where $f_0(z) = f(z)-b_0(z)$.
\item  For every positive integer $n$, and for some function (or all functions) $g_n(z)$ with $g_n(z) = o(1) \ (x \to 0)_\XX$ and $a_{n+1}(z) = o(g_n(z)) \ (x \to 0)_\XX$,  the function  
\begin{align*}
f(z) -   b_0(z)-  \frac{a_1(z)}{b_1(z) \, -} \  \frac{a_2(z)}{b_2(z) \,-} \  \frac{a_3(z)}{b_3(z) \,-} \ \cdots \   \frac{a_{n}(z)}{b_{n}(z) \, -} \ \frac{g_n(z)}{1}
\end{align*}
is asymptotic over $\XX$ at $0$ to $-\frac{a_1(z) a_2(z) \cdots a_n(z)}{(\beta_1\beta_2\cdots \beta_n)^2} g_n(z)$.
\end{enumerate}
\end{enumerate}
\end{theorem}

\begin{proof}
We may suppose without loss  of generality that $b_0(z) = 0$.  The numerator $A_n = A_n(z)$ and denominator $B_n = B_n(z)$ of $w_n(z) = \frac{ A_n(z)}{B_n(z)}$ are uniquely determined by the well-known recurrence relations \cite[(1.4)]{wall}  and therefore by induction are polynomials, with
$$B_n(z) = b_1(z)b_2(z)\cdots b_n(z) + \, \text{terms each involving some } a_i(z)$$
and
$$A_n(z) = a_1(z)(b_2(z)b_3(z)\cdots b_n(z)+ \, \text{terms each involving some } a_i(z))$$
for all $n$.  Therefore
$$B_n(z) = \beta_1\beta_2\cdots \beta_n+z D_n(z) \quad \text{ and } \quad A_n(z)= a_1(z)(\beta_2\beta_3\cdots \beta_n+ z C_n(z))$$
for some $C_n(z), D_n(z) \in \CC[z]$, and thus $w_n(z) \sim  \frac{a_1(z)\beta_2\beta_3\cdots \beta_n}{\beta_1\beta_2\cdots \beta_n}  = \frac{a_1(z)}{\beta_1} \  (z \to 0).$
Moreover, by \cite[(42.9)]{wall}, one has 
\begin{align*}
w_n(z)-w_{n-1}(z) = \frac{a_1(z) a_2(z) \cdots a_n(z)}{B_n(z)B_{n-1}(z)} = \frac{a_1(z) a_2(z) \cdots a_n(z)}{(\beta_1\beta_2\cdots \beta_n+z D_n(z))(\beta_1\beta_2\cdots \beta_{n-1}+z D_{n-1}(z))},
\end{align*}
and therefore
\begin{align*}
w_n(z)-w_{n-1}(z) 
= a_1(z) a_2(z) \cdots a_n(z)\left(\frac{1}{(\beta_1\beta_2\cdots \beta_{n-1})^2\beta_n} + zG_n(z)\right)
\end{align*}
for some $G_n(z) \in \CC[[z]]$.  It follows that 
\begin{align*}
w_n(z)-w_{n-1}(z) \equiv \frac{a_1(z) a_2(z) \cdots a_n(z)}{(\beta_1\beta_2\cdots \beta_{n-1})^2\beta_n}  \ \left(\operatorname{mod} \, \left(z^{m_1+m_2+\cdots+m_n+1}\right)\right)
\end{align*}
so the continued fraction in the theorem converges $(z)$-adically to a series $F(z) = \sum_{k = 1}^\infty (w_k(z)-w_{k-1}(z))$ in  $\CC[[z]]$.  Moreover, since also $G_n(z)$ is a rational function, it is analytic at $0$, and thus
\begin{align*}
w_n(z)-w_{n-1}(z) - \frac{a_1(z) a_2(z) \cdots a_n(z)}{(\beta_1\beta_2\cdots \beta_{n-1})^2\beta_n}  = O \left(z^{m_1+m_2+\cdots+m_n+1}\right) \ (z \to 0)
\end{align*}
Statement (1) follows.  

By statement (1) and Remark \ref{crem}(1), statements (2)(b)--(d) are equivalent.  Now, one has 
\begin{align*}
F(z)-w_n(z) =  \sum_{k = n+1}^\infty (w_k(z)-w_{k-1}(z)) \equiv w_{n+1}(z)-w_{n}(z) \  \left( \operatorname{mod} \, \left(z^{m_1+m_2+\cdots+m_{n+1}+1}\right)\right),
\end{align*} and therefore $F(z) - w_n(z) = z^{m_1+m_2+\cdots+ m_{n+1}}F_n(z)$
for some invertible $F_n(z) \in \CC[[z]]$.  Thus  one has $$\sum_{k = 1}^{m_1+m_2+\cdots+ m_{n+1}-1} c_kz^k - w_n(z)  = z^{m_1+m_2+\cdots+ m_{n+1}}H_n(z)$$
for some $H_n(z) \in  \CC[[z]]$.  Since then $H_n(z)$ is also a rational function of $z$,  it follows that  $H_n(z)$ is analytic at $0$ and thus
$$\sum_{k = 1}^{m_1+m_2+\cdots+ m_{n+1}-1}  c_kz^k - w_n(z)  = O\left( z^{m_1+m_2+\cdots+ m_{n+1}} \right) \ (z \to 0).$$  Therefore, one has
$$f(z) - \sum_{k = 1}^{m_1+m_2+\cdots+ m_{n+1}-1}c_k z^k  = O\left( z^{m_1+m_2+\cdots+ m_{n+1}} \right) \ (z \to 0)_{
\XX}, \quad \text{ for all } n \geq 0$$
if and only if 
$$f(z) - w_n(z)  = O\left( z^{m_1+m_2+\cdots+ m_{n+1}} \right) \ (z \to 0)_{
\XX}, \quad \text{ for all } n \geq 0.$$
This proves that statements (2)(a) and (2)(d) are equivalent.

Next, we verify the equivalence of (2)(c) and (2)(e).  By \cite[(1.3)]{wall} (or \cite[(1.3.2)]{cuyt}) and the definition of the function $f_n(z)$ one has
\begin{align}\label{fn}
f(z) = \frac{A_n(z)-A_{n-1}(z)f_n(z)}{B_n(z)-B_{n-1}(z)f_n(z)}
\end{align}
and therefore by the determinant formula   \cite[(1.5)]{wall}  \cite[(1.3.4)]{cuyt} one has
\begin{align}\label{AB}
f(z) - w_n(z) = \frac{A_n(z)-A_{n-1}(z)f_n(z)}{B_n(z)-B_{n-1}(z)f_n(z)}- \frac{ A_n(z)}{B_n(z)} =  \frac{a_1(z) a_2(z) \cdots a_n(z) f_n(z)}{B_n(z)(B_n(z)-B_{n-1}(z)f_n(z))}.
\end{align}
Thus, if (2)(e) holds, then 
$$f(z) - w_n(z) \sim  \frac{a_1(z) a_2(z) \cdots a_{n}(z)\frac{a_{n+1}(z)}{\beta_{n+1}}}{(\beta_1\beta_2\cdots \beta_{n})(\beta_1\beta_2\cdots \beta_n)}  =  \frac{a_1(z) a_2(z) \cdots a_{n+1}(z)}{(\beta_1\beta_2\cdots \beta_n)^2\beta_{n+1}},$$
whence (2)(c) holds.   Conversely, inverting (\ref{fn}), we have
\begin{align*}
f_n(z) = \frac{A_n(z)-B_n(z)f(z)}{A_{n-1}(z)-B_{n-1}(z)f(z)} = \frac{f(z)-w_n(z)}{f(z)-w_{n-1}(z)}\cdot \frac{B_n(z)}{B_{n-1}(z)},
\end{align*}
so, if (2)(c) holds, then
\begin{align*}
f_n(z) = \frac{f(z)-w_n(z)}{f(z)-w_{n-1}(z)}\cdot \frac{B_n(z)}{B_{n-1}(z)} \sim \frac{\frac{a_1(z) a_2(z) \cdots a_{n+1}(z)}{ ( \beta_1\beta_2\cdots \beta_n)^2 \beta_{n+1} }}{\frac{a_1(z) a_2(z) \cdots a_{n}(z)}{ (\beta_1\beta_2\cdots \beta_{n-1})^2 \beta_n }} \cdot \frac{\beta_1\beta_2\cdots \beta_n}{\beta_1\beta_2\cdots \beta_{n-1}} = \frac{a_{n+1}(z)}{\beta_{n+1}} \ (z \to \infty)_{\XX},
\end{align*} 
 whence (2)(e) holds.  Therefore, conditions (2)(c) and (2)(e) are equivalent.

Finally, we show that (2)(c) and (2)(f) are equivalent.   Let $$G_n(z) =  b_0(z)+ \frac{a_1(z)}{b_1(z) \, -} \  \frac{a_2(z)}{b_2(z) \,-} \  \frac{a_3(z)}{b_3(z) \,-} \ \cdots \   \frac{a_{n}(z)}{b_{n}(z) \,-} \  \frac{g_n(z)}{1}.$$  As in the proof of (\ref{AB}), one has
$G_n(z) - w_n(z) =  \frac{a_1(z)a_2(z) \cdots a_n(z) g_n(z)}{B_n(z)(B_n(z)-B_{n-1}(z)g_n(z))}$,
and therefore
$$G_n(z) - w_n(z)  \sim \frac{a_1(z) a_2(z) \cdots a_n(z) }{(\beta_1 \beta_2 \cdots \beta_n)^2 }g_n(z) =   o(w_n(z)-w_{n-1}(z)) \ (z \to 0)_\XX$$
Thus, if (2)(c) holds, then $$f(z) -w_n(x) \sim \frac{a_1(z) a_2(z) \cdots a_n(z) }{(\beta_1 \beta_2 \cdots \beta_n)^2 \beta_{n+1}}a_{n+1}(z) = o(- (G_n(z) - w_n(z)) \ (z \to 0)_\XX$$
and therefore
$$f(z)-G_n(z) \sim   - (G_n(z) - w_n(z))  \sim -  \frac{a_1(z) a_2(z) \cdots a_n(z) }{(\beta_1 \beta_2 \cdots \beta_n)^2 }g_n(z)\ (z \to 0)_{\XX},$$   
so that (2)(f) holds.  Conversely, if (2)(f) holds, then
$$f(z)-G_n(z) \sim  -\frac{a_1(z) a_2(z) \cdots a_n(z) }{(\beta_1 \beta_2 \cdots \beta_n)^2 }g_n(z) \sim  -(G_n(z)-w_n(z))  \ (z \to 0)_\XX$$
and therefore
$$f(z) - w_n(z) = o ( -(G_n(z)-w_n(z)) ) = o\left( w_n(z)-w_{n-1}(z)\right) \ (z \to 0)_\XX,$$
so that (2)(b) holds.  Therefore, conditions (2)(b), (2)(c), and (2)(f) are equivalent.
\end{proof}

An  obvious ``inverse'' equivalent of Theorem \ref{wsimG},  obtained by replacing $z$ with $1/z$ and $0$ with $\infty$,  shows that any asymptotic continued fraction expansion at $\infty$ with respect to a continued fraction that has nonzero terms $a_n(z)$ in $(1/z)\CC[1/z]$ and $b_n(z)$ in $\CC[1/z]\backslash (1/z)\CC[1/z]$ is equivalent to an  asymptotic expansion with respect to the asymptotic sequence $\left\{\frac{1}{z^n}\right\}$  at $\infty$.  Since Theorem \ref{wsimG} applies to any {\it associated continued fraction} \cite[p.\ 36]{cuyt},
 its ``inverse'' equivalent applies to any  Jacobi continued fraction.  Thus, Theorem \ref{wsimG} yields the equivalence of conditions (2)(a)--(d), (2)(f), and (2)(g) of Theorem \ref{wsimJ}.  Condition (2)(e), however, is unique to Jacobi continued fractions.

Theorem  \ref{wsimG} also applies to any {\it C-fraction},  {\it T-fraction}, or {\it M-fraction} \cite[pp.\ 35--38]{cuyt}.   Notably, a {\bf C-fraction} is a continued fraction of the form
\begin{align}\label{CF}
 G(z) = a_0+\frac{a_1 z^{m_1}} {1 \, -} \ \frac{a_2z^{m_2}} {1 \, -}  \ \frac{a_3z^{m_3}} {1 \, -} \ \frac{a_4z^{m_4}} {1 \, -} \  \cdots,
\end{align}
where $a_n \in \CC$ and $m_n \in \ZZ_{> 0}$ for all  $n$.  Any C-fraction $G(z)$ converges formally to a  series $F(z) \in \CC[[z]]$.  Conversely, any  series $F(z) \in \CC[[z]]$ can be written as the limit of a unique  C-fraction $G(z)$, called the {\bf C-fraction expansion of $F(z)$}.  
  The C-fraction expansion $G(z)$ of $F(z)$ is terminating, i.e., $a_n = 0$ for some $n$, if and only if $F(z)$ is a rational function in $\CC(z)$.   For constructive proofs of these assertions, see \cite[Chapter V]{lor}.


 A {\bf regular C-fraction} is a C-fraction (\ref{CF})  with $m_n = 1$ for all $n$.   A {\bf modified C-fraction} is a continued fraction of the form $G(1/z)$ for some C-fraction $G(z)$.    Theorem \ref{wsimG} and \cite[Chapter V Theorem 5]{lor} imply that any asymptotic expansion  of the form $f(z) \, \simeq \, \sum_{n = 0}^\infty c_nz^n \ (z \to 0)_{\XX}$, where $ \sum_{n = 0}^\infty c_n z^n \in \CC[[z]]\backslash \CC(z)$, has a unique equivalent formulation as an asymptotic C-fraction expansion over $\XX$ at $0$, and vice versa.  (Thus, any non-rational  real or complex function that is analytic at $0$ has a unique asymptotic C-fraction expansion at $0$.)  Equivalently, any asymptotic expansion of the form $f(z) \, \simeq \, \sum_{n = 0}^\infty \frac{c_n}{z^n} \ (z \to \infty)_{\XX}$, where $ \sum_{n = 0}^\infty c_n z^n  \in \CC[[z]]\backslash \CC(z)$, has a unique equivalent formulation as an asymptotic modified C-fraction expansion over $\XX$ at $\infty$, and vice versa.

\section{Asymptotic continued fraction expansions of $\pi(x)$}

\subsection{Proof of Theorem \ref{maincontthm1}}

We now use (\ref{asex2}), Corollary \ref{wsim}, and a result of Stieltjes to prove Theorem \ref{maincontthm1}. 

\begin{proof}[Proof of Theorem \ref{maincontthm1}]
It is known that the C-fraction expansion of the formal power series $F(z) = \sum_{n = 0}^\infty n!z^{n+1}$ in $\CC[[z]]$ is given by
$\frac{z} {1 \, -} \ \frac{z} {1 \, -}  \ \frac{z} {1 \, -} \ \frac{2z} {1 \, -}  \ \frac{2z} {1 \, -} \ \frac{3z} {1 \, -} \ \frac{3z} {1 \, -}  \  \cdots$.
Indeed, this was first proved by Stieltjes in \cite[No.\ 57]{stie}, from Hankel matrix determinant expressions he obtained in \cite[No.\ 11]{stie} for the terms of a regular C-fraction in terms of its formal power series expansion in $\CC[[z]]$.  As a consequence, since $\PP(e^x)$  has the asymptotic expansion (\ref{pex}), the theorem follows immediately from the equivalence of statements (2)(a)--(c) of Corollary \ref{wsim} (and Remark \ref{crem}(5)).
\end{proof}

As a corollary of Theorems \ref{maincontthm1} and \ref{wsimJ}, we obtain the following.

\begin{corollary}\label{maincor}
The best rational approximations of the function $\PP(e^x)$ are precisely the approximants $w_n(x)$ of the continued fraction $$\frac{1}{x-1 \,-} \  \frac{1}{x-3 \,-} \  \frac{4}{x-5 \,-}\  \frac{9}{x-7 \,-}\  \frac{16}{x-9 \,-} \  \cdots.$$  Moreover, for all $n \geq 0$, one has $\PP(e^x)- w_n(x) \sim \frac{(n!)^2}{x^{2n+1}}  \ (x \to \infty)$, and $w_n(x)$ is the unique rational function of degree at most $n$ such that $\PP(e^x) - w_n(x) = O\left( \frac{1}{x^{2n+1}}\right) \ (x \to \infty)$.  Furthermore, for all $n \geq 1$ and all $a \in \RR$, the function
\begin{align*}
\PP(e^x) -  \frac{1}{x-1 \,-} \  \frac{1}{x-3 \,-} \  \frac{4}{x-5 \,-} \ \cdots \   \frac{(n-2)^2}{x-(2n-3)\, -} \  \frac{(n-1)^2}{x -(2n-1)+a}
\end{align*}
is asymptotic to $\frac{((n-1)!)^2}{x^{2n}}a$ if $a \neq 0$, and to $\frac{(n!)^2}{x^{2n+1}}$ if $a = 0$.
\end{corollary}




\begin{remark}\label{Laguerre} \
\begin{enumerate}
\item The numerator $P_n(x)$ and denominator  $Q_n(x)$ of the rational function $w_n(x)$ in Corollary \ref{maincor} are monic integer polynomials of degree $n-1$ and $n$, respectively, and by (\ref{detformula}) one has $w_{n+1}(x)-w_{n}(x) = \frac{(n!)^2}{Q_{n+1}(x)Q_{n}(x)}$  for all $n \geq 0$.     Laguerre showed in \cite{lag} that the denominator $Q_n(x)$ of $w_n(x)$ is given by $Q_n(x) = \widehat{L}_n(x)$, where $$\widehat{L}_n(x) = (-1)^n n! L_n(x) = \sum_{k = 0}^n (-1)^k k!{n\choose k}^2 x^{n-k}$$
is the $n$th {\bf monic Laguerre polynomial} and
$$L_n(x) = \sum_{k = 0}^n \frac{(-1)^k}{k!}{n\choose k} x^k = {}_{1}F_{1}(-n;1;x)$$
is the  {\bf $n$th Laguerre polynomial}.  The polynomial $L_n(x)$ is known to have $n$ distinct positive real roots, and its largest root lies between $3n-4$ and $4n+2$ \cite{sko}.  It follows that there is no fixed neighborhood of $\infty$ on which  all of the functions $w_n(x)-w_{n-1}(x)$ are defined.
\item From  \cite[{[1.14]}]{akh}, one can deduce that the numerator $P_n(x)$ of $w_n(x) = \frac{P_n(x)}{\widehat{L}_n(x)}$ is given by
$P_n(x) =  \sum_{k = 0}^{n-1} a_{n,k} x^{n-1-k},$
where $$a_{n,k} = \sum_{j = 0}^k  {(-1)^{j}}j!(k-j)! {n \choose j}^2=  (-1)^{k} k! {n \choose k}^2 {}_{3}F_{2}(1,1,-k; n-k+1, n-k+1; 1)$$
for all $n \geq 1$ and $1 \leq k \leq n$,
with explicit values $a_{n,0} = 1$,  $a_{n,1} = 1-n^2$, $a_{n,n-1} = (-1)^{n+1}n!H_n = s(n+1,2)$, and $a_{n,n-2} = (-1)^n n!((n+1)H_n-2n) =(-1)^n\langle \langle n,n-2 \rangle \rangle$, where $H_n = \sum_{k = 1}^n \frac{1}{k}$ denotes the $n$th harmonic number, $s(n,k)$ denotes the signed Stirling number of the first kind, and $\langle \langle n,k \rangle \rangle$ denotes the Eulerian number of the second kind.
\end{enumerate}
\end{remark}

\subsection{Proof of Theorem \ref{gentheorem}}

We now use Corollary \ref{wsim} to prove Theorem \ref{gentheorem}.    For $a, b \in \CC$, consider the  hypergeometric series $${}_{2}F_{0}(a,b;;z) = \sum_{k = 0}^\infty \frac{(a)_k(b)_k}{n!}z^k$$
in $\CC[[z]]$, where $(x)_n = x(x+1)(x+2)\cdots(x-n+1)$ denotes the {\bf Pochhammer symbol}.  
 In \cite[(89.5) and (92.2)]{wall} it is proved that ${}_{2}F_{0}(a,1;;z)$ for all $a \in \CC$ has the formal C-fraction expansion
$$\sum_{k = 0}^\infty (a)_k z^k = {}_{2}F_{0}(a,1;;z)  = \frac{1} {1 \, -} \ \frac{a z} {1 \, -}  \ \frac{1z} {1 \, -} \ \frac{(1+a)z} {1 \, -} \   \frac{2z} {1 \, -} \ \frac{(2+a)z} {1 \, -} \   \frac{3z} {1 \, -} \ \frac{(3+a)z} {1 \, -} \cdots.$$

\begin{proof}[Proof of Theorem \ref{gentheorem}] 
By the  C-fraction expansion above, formally in $\CC[[1/z]]$ one has
$$\sum_{k = 0}^\infty\frac{ (k+n)! }{n!}\frac{1}{z^{k+1}}  = \frac{1}{z}{}_{2}F_{0}(n+1,1;;1/z)  = \cfrac{\frac{1}{z}} {1 \, -} \ \cfrac{\frac{1+n}{z}} {1 \, -}  \ \cfrac{\frac{1}{z}} {1 \, -} \ \cfrac{\frac{2+n}{z}} {1 \, -} \   \cfrac{\frac{2}{z}} {1 \, -} \ \cfrac{\frac{3+n}{z}} {1 \, -} \   \cfrac{\frac{3}{z}} {1 \, -} \ \cfrac{\frac{4+n}{z}} {1 \, -} \cdots.$$
Also, from the asymptotic expansion (\ref{pex}) of $\PP(e^x)$
we easily obtain the asymptotic expansion
\begin{align*}
p_n(e^x) \simeq \sum_{k = 0}^\infty \frac{(k+n)!}{n!} \frac{1}{x^{k+1}} \ (x \to \infty),
\end{align*}
and likewise fron (\ref{asex}) we obtain the same asymptotic expansion for $l_n(e^x)$.
As a consequence, the theorem follows from  the equivalence of statements (2)(a)--(c) of Corollary \ref{wsim}.
\end{proof}

\begin{remark}
By \cite[pp.\ 95--96]{cuyt}, the monic denominator of the $k$th approximant of the Jacobi continued fraction in Theorem \ref{gentheorem} is equal to $\widehat{L}_k^{(n)}(\log x)$, where $\widehat{L}_k^{(n)}(z)$ denotes the {\it $k$th monic  generalized Laguerre  polynomial}.
\end{remark}

To provide further context for Theorem \ref{gentheorem}, we note that the fundamental asymptotic expansion  (\ref{asex2}) of $\PP(x)$ is by definition equivalent to  the asymptotic $p_n(x) \sim \frac{1}{\log x} \ (x \to \infty)$ holding for all $n \geq 0$, and the theorem for $n = 1$ yields the following.

\begin{corollary}\label{Plog}
One has the asymptotic continued fraction expansions
$$\PP(x)\log x  \,  \simeq \, 1+ \cfrac{\frac{1}{\log x}}{1 \,-} \ \cfrac{\frac{2}{\log x}}{1 \,-}\  \cfrac{\frac{1}{\log x}}{1 \,-}\  \cfrac{\frac{3}{\log x}}{1 \,-}\  \cfrac{\frac{2}{\log x}}{1 \,-}\  \cfrac{\frac{4}{\log x}}{1 \,-} \  \cfrac{\frac{3}{\log x}}{1 \,-}\  \cfrac{\frac{5}{\log x}}{1 \,-} \  \cfrac{\frac{4}{\log x}}{1 \,-}  \ \ \cdots \ (x \to \infty)$$
and
$$\PP(x)\log x  \,  \simeq \, 1+ \frac{1}{\log x -2\,-} \  \frac{1\cdot 2}{\log x - 4 \,-}\  \frac{2\cdot 3}{\log x-6 \,-}\  \frac{3 \cdot 4}{\log x - 8 \,-}  \  \frac{4 \cdot 5}{\log x - 10 \,-} \ \cdots \  (x \to \infty).$$
\end{corollary}

 By \cite[p.\ 243]{cuyt}, for all $a \in \CC$ the series $\frac{1}{z}{}_{2}F_{0}(a,1;;1/z) \in \CC[[1/z]]$ has  formal T-fraction expansion 
\begin{align*}
\frac{1}{z}{}_{2}F_{0}(a,1;;1/z) &  = \cfrac{\frac{1}{z}}{1+\frac{1-a}{z} \, -}  \ \cfrac{\frac{1}{z}} {1+\frac{2-a}{z}\, -} \ \cfrac{\frac{2}{z}} {1+\frac{3-a }{z}\, -} \   \cfrac{\frac{3}{z}} {1+\frac{4-a}{z} \, -} \ \cfrac{\frac{4}{z}}{1+\frac{5-a}{z} \, -}  \  \cdots  \\
&  = \frac{1}{z+1-a \, -}  \ \frac{1z} {z+2-a \, -} \ \frac{2z} {z+3-a \, -} \   \frac{3z} {z+4-a \, -} \ \frac{4z}{z+5-a \, -}  \  \cdots.
\end{align*}
Thus, from the ``inverse'' equivalent of Theorem  \ref{wsimG}, and as in the proof of Theorem \ref{gentheorem}, we obtain the following result,  which is a generalization of Proposition \ref{gencornew} and an analogue of Theorem \ref{gentheorem}.

\begin{theorem}\label{gentheoremnew}
For every nonnegative integer $n$ one has the asymptotic continued fraction expansion
$$p_n(x) \, \simeq \,  \cfrac{\frac{1}{\log x}}{1+\frac{-n}{\log x} \, -}  \ \cfrac{\frac{1}{\log x}} {1+\frac{1-n}{\log x}\, -} \ \cfrac{\frac{2}{\log x}} {1+\frac{2-n }{\log x}\, -} \   \cfrac{\frac{3}{\log x}} {1+\frac{3-n}{\log x} \, -} \ \cfrac{\frac{4}{\log x}}{1+\frac{4-n}{\log x} \, -}  \  \cdots \ (x \to \infty),$$
or, equivalently,
$$p_n(x)  \,  \simeq \,  \frac{1}{\log x-n \, -}  \ \frac{1\log x} {\log x+1 -n\, -} \ \frac{2\log x} {\log x+2-n \, -} \   \frac{3\log x} {\log x+3-n\, -} \ \frac{4\log x}{\log x+4-n \, -}  \  \cdots \  (x \to \infty),$$
where  $p_n(x) = \frac{(\log x)^{n}}{n!} \left(\PP(x) - \sum_{k = 0}^{n-1} \frac{k!}{(\log x)^{k+1}}\right).$  Moreover, the same asymptotic continued fraction expansions hold for the function $l_n(x) = \frac{(\log x)^{n}}{n!} \left(\frac{\li(x)}{x} - \sum_{k = 0}^{n-1} \frac{k!}{(\log x)^{k+1}}\right).$
\end{theorem}

Note that the two  continued fractions in the theorem are equivalent, and, if $w_{n,k}(x)$ denotes the $k$th approximant, then by  Theorem  \ref{wsimG}   one has $p_n(x) - w_{n,k}(x) \sim \frac{k!}{(\log x)^k} \ (x \to \infty)$ for all $n$ and $k$.  For $n = 0$, the theorem yields Proposition \ref{gencornew}.  For $n = 1$, it yields the following.

\begin{corollary}\label{rkr}
One has the asymptotic continued fraction expansion
$$\PP(x)\log x \, \simeq \, 1+ \cfrac{\frac{1}{\log x}}{1-\frac{1}{\log x} \, -}  \ \cfrac{\frac{1}{\log x}} {1\, -} \ \cfrac{\frac{2}{\log x}} {1+\frac{1 }{\log x}\, -} \   \cfrac{\frac{3}{\log x}} {1+\frac{2}{\log x} \, -} \ \cfrac{\frac{4}{\log x}}{1+\frac{3}{\log x} \, -}  \  \cdots \ (x \to \infty),$$
or, equivalently,
$$\PP(x)\log x \, \simeq \, 1+ \frac{1}{\log x-1 \, -}  \ \frac{1\log x} {\log x \, -} \ \frac{2\log x} {\log x+1 \, -} \   \frac{3\log x} {\log x+2\, -} \ \frac{4\log x}{\log x+3 \, -}  \  \cdots \  (x \to \infty).$$
\end{corollary}

 Corollary \ref{rkr} also follows from (\ref{asex2}) and the fact, proved by induction, that the $n$th approximant of both continued fractions in the corollary is equal to $\frac{n \cdot n!}{(\log x - n)(\log x)^n} + \sum_{k = 0}^n \frac{k!}{(\log x)^k}$.

\subsection{Measure-theoretic interpretation of Stietljes and Jacobi continued fractions}

In this section we consider the meausure-theoretic aspects of Stieltes' theory of continued fractions  \cite{stie} and their consequences for asymptotic Stieltjes and Jacobi continued fraction expansions.  

Let $\mu$ be a  (positive) measure on $\RR$.  For all integers $k$, the {\bf $k$th moment of $\mu$} is the integral
$$m_k(\mu) = \int_{-\infty}^\infty t^k d \mu(t).$$  In its modern formulation, the {\bf Stieltjes moment problem}, posed and motivated in \cite[No.\ 24]{stie} by Stieltjes in connection with his extensive theory of continued fractions \cite{stie}, is the problem of determining for which sequences $\{\mu_k\}_{k = 0}^\infty$ of real numbers there exists a Borel measure $\mu$ on $[0, \infty)$ such that $\mu_k = m_k(\mu)$ for all nonnegative integers $k$.  To solve this problem Stieltjes introduced what we now call the {\bf Stieltjes transform} of $\mu$, which is the complex function
$$\SS_\mu(z) = \int_{-\infty}^\infty \frac{d \mu(t)}{z-t}.$$   
If the measure $\mu$ is finite, then $\SS_\mu(z)$ is analytic on $\CC\backslash \operatorname{supp} \mu$ with derivative  $\frac{d}{dz}\SS_\mu(z)  = -\int_{-\infty}^\infty \frac{d \mu(t)}{(z-t)^2}$.  Moreover, Stieltjes established in \cite{stie} the following remarkable result.   

\begin{theorem}[\cite{stie} {\cite[Chapter VII Theorems 3 and 4]{lor}} {\cite[Theorems 5.1.1 and 5.2.1]{cuyt}}]\label{mutheorem1}
Let $\{\mu_k\}_{k = 0}^\infty$ be a sequence of real numbers.   There exists a Borel measure $\mu$ on $[0,\infty)$ with infinite support (i.e., that is not a finite sum of point masses) such that $\mu_k = m_k(\mu)$ for all nonnegative integers $k$ if and only if there exists a Stieltjes continued fraction
\begin{align*}
G(z) = \cfrac{\frac{a_1}{z} }{1 \, -} \ \cfrac{\frac{a_2}{z}} {1 \, -}  \ \cfrac{\frac{a_3}{z}} {1 \, -} \ \cfrac{\frac{a_4}{z}} {1 \, -} \  \cdots \ = \ \cfrac{a_1} {z \, -} \ \cfrac{a_2} {1 \, -}  \ \cfrac{a_3} {z \, -} \ \cfrac{a_4} {1 \, -} \  \cdots,
\end{align*}
where $a_n \in \RR_{> 0}$ for all $n$, such that $G(z)$ converges $(1/z)$-adically  in $\CC[[1/z]]$ to the series  $\sum_{k = 0}^\infty\frac{ \mu_k}{ z^{k+1}}$. 
If these conditions hold, then
$\SS_\mu(z)$ is analytic on $\CC\backslash [0,\infty)$ and for all $\varepsilon > 0$ has the asymptotic expansion
$\SS_\mu(z) \simeq \sum_{k = 0}^\infty \frac{\mu_k}{z^{k+1}} \ (z \to \infty)_{\CC_\varepsilon}$
over  ${\CC_\varepsilon} = \{z \in \CC: |\operatorname{Arg}(z)| \geq \varepsilon\}$.  Moreover, the Borel measure $\mu$ is unique if and only if the continued fraction  $G(z)$ converges for some $z \in \CC$, in which case it converges for all $z \in \CC \backslash [0,\infty)$ and  $\SS_\mu(z) = G(z)$ for all $z \in \CC\backslash [0,\infty)$.
\end{theorem}

Stieltjes' theorem above and Corollary \ref{wsim} have the following consequence for asymptotic Stieltjes continued fraction expansions.

\begin{theorem}\label{mutheorem}
Let $\mu$ be a Borel measure on $[0, \infty)$ with infinite support and finite moments, let $f$ be a complex function, and let $\XX$ be an unbounded subset of $\CC$.    Then $f$ has the asymptotic expansion
$$f(z) \simeq \sum_{k = 0}^\infty \frac{m_k(\mu)}{z^{k+1}} \ (z \to \infty)_\XX$$
if and only if $f$ has the  asymptotic continued fraction expansion
$$f(z) \, \simeq \,  \cfrac{\frac{a_1}{z}} {1 \, -} \ \cfrac{\frac{a_2}{z}} {1 \, -}  \ \cfrac{\frac{a_3}{z}} {1 \, -} \ \cfrac{\frac{a_4}{z}} {1 \, -}  \  \cdots \ (z \to \infty)_\XX,$$
where the $a_n \in \RR_{> 0}$ are as in Theorem \ref{mutheorem1}. 
\end{theorem}

\begin{proof}
The theorem  follows immediately Theorem \ref{mutheorem1} and the equivalence of statements (2)(a) and (2)(b) of Corollary  \ref{wsim}. 
\end{proof}


One also has the following analogues  of Theorems \ref{mutheorem1} and \ref{mutheorem}  for Borel  measures on $\RR$ and Jacobi continued fractions.  Theorem  \ref{mutheorem1bb}  is due to Hamburger \cite{ham}

\begin{theorem}[\cite{ham} {\cite[Theorems 5.1.3 and 5.2.3]{cuyt}}]\label{mutheorem1bb} 
Let $\{\mu_k\}_{k = 0}^\infty$ be a sequence of real numbers. 
There exists a Borel measure $\mu$ on $\RR$  with infinite support such that $\mu_k = m_k(\mu)$ for all nonnegative integers $k$ if and only if there exists a Jacobi continued fraction
\begin{align*}
G(z) = \frac{a_1}{z+b_1 \,-} \  \frac{a_2}{z+b_2 \,-} \  \frac{a_3}{z+b_3 \,-}  \ \cdots,
\end{align*}
where $a_n \in \RR_{> 0}$ and $b_n  \in \RR$ for all $n$, that converges $(1/z)$-adically  in $\CC[[1/z]]$ to the series $\sum_{k = 0}^\infty \frac{\mu_k }{z^{k+1}}$.  If these conditions hold, then
$\SS_\mu(z)$ is analytic on $\CC\backslash \RR$ and for all $\delta, \varepsilon > 0$ has the asymptotic expansion
$\SS_\mu(z) \simeq \sum_{k = 0}^\infty \frac{\mu_k}{z^{k+1}} \ (z \to \infty)_{\CC_{\delta, \varepsilon}}$ over  $\CC_{\delta, \varepsilon} = \{z \in \CC : \delta \leq |\operatorname{Arg}(z)| \leq \pi- \varepsilon\}$, and one has $\SS_\mu(z) =  G(z)$ for all $z \in \CC\backslash \RR$.
\end{theorem}

\begin{theorem}\label{mutheorem1cc}
Let $\mu$ be a  Borel measure on $\RR$ with infinite support and finite moments, let $f$ be a complex function, and let $\XX$ be an unbounded subset of $\CC$.   Then $f$ has the asymptotic expansion
$$f(z) \simeq \sum_{k = 0}^\infty \frac{m_k(\mu)}{z^{k+1}} \ (z \to \infty)_\XX$$
if and only if $f$ has the  asymptotic continued fraction expansion
$$f(z) \, \simeq \,\frac{a_1}{z+b_1 \,-} \  \frac{a_2}{z+b_2 \,-} \  \frac{a_3}{z+b_3 \,-} \  \cdots  \ (z \to \infty)_\XX,$$
where the $a_n \in \RR_{> 0}$ and $b_n \in \RR$ are as in Theorem \ref{mutheorem1bb}.  
\end{theorem}

\begin{proof}
The theorem  follows immediately Theorem \ref{mutheorem1bb}  and the equivalence of statements (2)(a) and (2)(b) of Theorem \ref{wsimJ}. 
\end{proof}

\subsection{Measure-theoretic proof  of Theorems \ref{maincontthm1} and \ref{gentheorem}}

In this section we use the results of the previous section to provide an alternative measure-theoretic proof of Theorems \ref{maincontthm1} and \ref{gentheorem}.  

Let $\gamma_0$ denote the probability measure on $[0,\infty)$ with density function $e^{-t}$, which is known as the {\bf exponential distribution with rate parameter $1$}.    The measure $\gamma_0$ is the unique Borel measure $\mu$ on $[0,\infty)$ that has $k$th moment $m_k(\mu) = \int_0^\infty t^k d\mu(t)$ equal to $k!$ for all $k \geq 0$.   Moreover, its density function $e^{-t}$ is the unique piecewise continuous function $\rho(t)$ on $[0,\infty)$ whose Mellin transform $\int_0^\infty t^{s-1}\rho(t)\, dt$ is equal to the gamma function $\Gamma(s)$. Furthermore, $\gamma_0$ is the unique Borel  measure $\mu$ on $[0,\infty)$ whose Stieltjes transform on $\CC\backslash [0,\infty)$ is equal to $-e^{-z}E_1(-z)$, where $E_1(z)$ is the {\bf exponential integral function}
$$E_1(z) = \int_{z}^\infty \frac{e^{-t}}{t} dt, \quad z \in \CC \backslash (-\infty, 0],$$
where the integral is along any path of integration not crossing $(-\infty, 0]$.  The function $E_1(z)$ is analytic on $\CC \backslash (-\infty,0]$, while $E_1(x) := \lim_{\varepsilon \to 0^+} E_1(x+\varepsilon i)$ and $\lim_{\varepsilon \to 0^-} E_1(x+\varepsilon i) = \overline{E_1(x)}$ for all $x < 0$.  The function $E_1(x)$ for nonzero real $x$ is given by
\begin{align}\label{E1x} E_1(x) =   \left.
  \begin{cases}
   -\li(e^{-x}) & \text{if } x > 0 \\
    -\li(e^{-x})-\pi i   & \text{if } x < 0.
 \end{cases}
\right.
\end{align}
In  \cite[No.\ 57]{stie}, Stieltjes proved that the Stieltjes transform $\SS_{\gamma_0}(z) = -e^{-z}E_1(-z)$ of $\gamma_0$ has the Stieltjes and Jacobi continued fraction expansions
\begin{align}
 -e^{-z}E_1(-z) & = \frac{1}{z \,-} \  \frac{1}{1 \,-} \  \frac{1}{z \,-}\  \frac{2}{1 \,-}\  \frac{2}{z \,-} \  \frac{3}{1 \,-}\  \frac{3}{z \,-} \ \cdots  \label{E11} \\
&  = \frac{1}{z-1 \,-} \  \frac{1}{z-3 \,-} \  \frac{4}{z-5 \,-}\  \frac{9}{z-7 \,-}\  \frac{16}{z-9 \,-} \  \cdots   \label{E12}
\end{align}
on $\CC \backslash [0,\infty)$.  

\begin{proof}[Proof of Theorem \ref{maincontthm1}]
Stieltjes proved in \cite[No.\ 57]{stie} that  $\gamma_0$ is the unique Borel measure associated as in Theorem \ref{mutheorem1} to the Stieltjes continued fraction (\ref{E11}).     The moments of  $\gamma_0$ are given by $m_k(\gamma_0) = k!$ for all $k$.  Therefore, by (\ref{asex2}), one has the asymptotic expansion $\PP(e^x) \simeq \sum_{k = 0}^{\infty} \frac{m_k(\gamma_0)}{x^{k+1}} \ (x \to \infty).$
The theorem therefore follows from Theorems \ref{mutheorem} and \ref{mutheorem1cc} and Remark \ref{crem}(5).
\end{proof}

To prove Theorem \ref{gentheorem}, we extend our analysis of the probability measure $\gamma_0$ to the probability measure $\gamma_n$ on $[0,\infty)$ with density function $\frac{t^{n}}{n!}e^{-t}$.  The measure $\gamma_n$ is the gamma distribution with  shape parameter  $n+1$  and rate parameter $1$ and is the $(n+1)$-fold convolution of the measure $\gamma_0$ with itself, where for all $a,b > 0$ the {\bf gamma distribution with shape parameter $a$ and and rate parameter $b$} is the  probability measure $\gamma_{a,b}$ on $[0,\infty)$ with density function $\frac{b^a}{\Gamma(a)}t^{a-1}e^{-bt}$, which has moments given by $m_k(\gamma_{a,b}) =  \frac{(a)_k}{b^{k}}$ for all $k$, where $(a)_k$ denotes the Pochhammer symbol.  By \cite[pp.\ 239--240]{cuyt}, the Stieltjes transform $\SS_{\gamma_{a,b}}(z)$ of $\gamma_{a,b}$ is equal to $-be^{-bz}E_{a}(-bz)$, where the {\bf exponential integral function} $E_a(z)$, for any $a \in \CC$, is the function
\begin{align*}
E_a(z) = z^{a-1}\int_{z}^\infty \frac{e^{-t}}{t^a} dt =  z^{a-1}\Gamma(1-a,z), \quad z \in \CC \backslash (-\infty,0],
\end{align*}
where $$\Gamma(s,z) = \int_{z}^\infty t^{s-1}e^{-t} dt, \quad s \in \CC, \quad z \in \CC \backslash (-\infty,0],$$ denotes the {\bf upper incomplete gamma function}, where the integrals are along any path of integration not crossing $(-\infty, 0]$ \cite[pp.\ 238 and 275]{cuyt}. 
Stieltjes proved in  \cite[No.\ 62]{stie} that the Stieltjes transform $\SS_{\gamma_{a,1}}(z) = -e^{-z}E_a(-z)$ of $\gamma_{a,1}$ has the continued fraction expansions
\begin{align}
-e^{-z}E_{a}(-z) & =  \cfrac{\frac{1}{z}}{1 \,-} \  \cfrac{\frac{a}{z}}{1 \,-}\  \cfrac{\frac{1}{z}}{1 \,-}\  \cfrac{\frac{1+a}{z}}{1 \,-}\  \cfrac{\frac{2}{z}}{1 \,-}\  \cfrac{\frac{2+a}{z}}{1 \,-} \ \cfrac{\frac{3}{z}}{1 \,-}\  \cfrac{\frac{3+a}{z}}{1 \,-} \ \cdots  \label{Eae}  \\
  &=  \frac{1}{z -a\,-} \  \frac{a}{z - 2-a \,-}\  \frac{2(1+a)}{z-4-a \,-}\  \frac{3(2+a)}{z -6-a \,-}  \ \frac{4(3+a)}{z -8-a \,-} \  \cdots  
\end{align}
on $\CC\backslash [0,\infty)$ (and  both continued fractions converge  formally in $\CC[[1/z]]$ to $(1/z){}_{2}F_{0}(a,1;;1/z)$).     It follows that  the Stieltjes transform $\SS_{\gamma_n}(z) = -e^{-z}E_{n+1}(-z)$ of the measure $\gamma_n = \gamma_{n+1,1}$ has the continued fraction expansions (\ref{Enexpansion}) and (\ref{Enexpansion2}) on $\CC\backslash [0,\infty)$.

\begin{proof}[Proof of Theorem \ref{gentheorem}]
By  \cite[No.\ 62]{stie}, the measure $\gamma_{a,1}$ for any $a > 0$ is the unique Borel measure associated as in Theorem \ref{mutheorem1} to the  continued fraction (\ref{Eae}), and thus $\gamma_n$ to the  continued fraction (\ref{Enexpansion}).   The moments  of $\gamma_n$ are given by $m_k(\gamma_n) =  \frac{(k+n)!}{n!}$ for all $k$.  But then, from the asymptotic expansion (\ref{pex}) of $\PP(e^x)$,
we easily obtain the asymptotic expansion $p_n(e^x) \simeq \sum_{k = 0}^\infty \frac{ m_k(\gamma_n) }{x^{k+1}} \ (x \to \infty)$, and likewise for $l_n(e^x)$.  The theorem therefore follows from Theorems \ref{mutheorem} and \ref{mutheorem1cc}.
\end{proof}



\begin{remark}
We   provide further context for the functions $l_n(x)$ in  Theorem \ref{gentheorem}, as follows.  Let $\mu$ be a finite measure on $\RR$.  Let $\SS_\mu(x+0^+i) = \lim_{\varepsilon \to 0^+}  \SS_\mu(x+\varepsilon i)$  for all $x \in \RR$ such that the limit exists.  In fact the limit exists for all $x \in \RR$ outside a set of Lebesgue measure zero, and  its real part by definition is equal to $\pi$ times the {\bf Hilbert transform}  ${\mathcal H}_\mu(x) = \frac{1}{\pi}  \operatorname{Re} \SS_\mu(x+0^+i)$ of the measure $\mu$ \cite{PSZ}. 
For all $a \in \CC$ and all $x <0$ one sets $E_a(x) := \lim_{\varepsilon \to 0^+} E_a(x+\varepsilon i)$, where also $\overline{E_{\overline{a}}(x)}  =  \lim_{\varepsilon \to 0^-} E_a(x+\varepsilon i)$.    By \cite[(14.1.9)]{cuyt}, one has
$-e^{-z}E_{n+1}(-z)  = \frac{z^n}{n!} \left(-e^{-z}E_{1}(-z)- \sum_{k = 0}^{n-1}\frac{k!}{z^{k+1}}\right)$
for all $z \in \CC \backslash \{0\}$,  which, along with (\ref{E1x}), implies that
$$\SS_{\gamma_n}(x+0^+i)  = \lim_{\varepsilon \to 0^+}\left( -e^{-(x+\varepsilon i)}E_{n+1}(-(x+\varepsilon i)) \right) = l_n(e^x)- \pi i \frac{x^{n}e^{-x}}{n! },$$
and therefore $$l_n(e^x) =  \operatorname{Re}(-e^{-x}E_{n+1}(-x)) =  \operatorname{Re}\SS_{\gamma_n}(x+0^+i) = \pi {\mathcal H}_{\gamma_n}(x),$$
for all $x > 0$. It is clear, more generally, that
$\SS_{\gamma_{a,1}}(x+0^+i)  =  -e^{-x}\overline{E_a(-x)}$ and therefore
$\pi {\mathcal H}_{\gamma_{a,1}}(x) = \operatorname{Re}\SS_{\gamma_{a,1}}(x+0^+i) = \operatorname{Re}(-e^{-x}E_a(-x))$  for all $a > 0$ and all $x> 0$.
\end{remark}

\section{Further asymptotic expansions of $\pi(x)$ and related functions}

\subsection{Asymptotic continued fraction expansions and harmonic numbers}

As a corollary of Theorem \ref{maincontthm1}, Corollary \ref{Plog}, and Lemmas \ref{asympprop} and \ref{asympprop2a}, we obtain the following.

\begin{proposition}\label{LP}
Let $L(x)$ be any function such that $L(x) = \log x + o ((\log x)^{-t}) \  (x \to \infty)$ for all $t > 0$, and let $P(x)$ be any function such that $P(x) = \frac{\li(x)}{x} + o ((\log x)^{-t}) \  (x \to \infty)$ for all $t > 0$.   For all $t \in \RR$, one has the following asymptotic continued fraction expansions.
\begin{enumerate}
\item  $\displaystyle P(e^t x) \,  \simeq \,  \cfrac{\frac{1}{L(x)+t}}{1 \,-} \ \cfrac{\frac{1}{L(x)+t}}{1 \,-}\  \cfrac{\frac{1}{L(x)+t}}{1 \,-}\  \cfrac{\frac{2}{L(x)+t}}{1 \,-}\  \cfrac{\frac{2}{L(x)+t}}{1 \,-}\  \cfrac{\frac{3}{L(x)+t}}{1 \,-} \  \cfrac{\frac{3}{L(x)+t}}{1 \,-}\  \cfrac{\frac{4}{L(x)+t}}{1 \,-}  \ \cfrac{\frac{4}{L(x)+t}}{1 \,-}  \ \cdots \ (x \to \infty)$.
\item  $\displaystyle P(e^t x)  \,  \simeq \, \frac{1}{L(x)+t -1\,-} \  \frac{1}{L(x)+t  - 3 \,-}\  \frac{4}{L(x)+t-5\,-}\  \frac{9}{L(x)+t - 7 \,-} \ \frac{16}{L(x)+t-9\,-}\   \cdots \  (x \to \infty).$
\item $\displaystyle {P(e^t x)(L(x)+t )}   \, \simeq \, 1+ \cfrac{\frac{1}{L(x)+t}}{1 \,-} \ \cfrac{\frac{2}{L(x)+t}}{1 \,-}\  \cfrac{\frac{1}{L(x)+t}}{1 \,-}\  \cfrac{\frac{3}{L(x)+t}}{1 \,-}\  \cfrac{\frac{2}{L(x)+t}}{1 \,-}\  \cfrac{\frac{4}{L(x)+t}}{1 \,-} \  \cfrac{\frac{3}{L(x)+t}}{1 \,-}\  \cfrac{\frac{5}{L(x)+t}}{1 \,-}  \ \cdots \ (x \to \infty).$
\item $\displaystyle {P(e^t x)(L(x)+t) }   \, \simeq \,  1+\frac{1}{L(x)+t-2 \,-} \  \frac{1\cdot 2}{L(x)+t-4 \,-}\  \frac{2 \cdot 3}{L(x)+t-6 \,-}\  \frac{3 \cdot 4}{L(x)+t-8 \,-}\ \cdots \ (x \to \infty)$.
\end{enumerate}
\end{proposition}


\begin{example}  The following provide examples of functions $L(x)$ satisfying the hypotheses of Proposition \ref{LP}.
\begin{enumerate}
\item  The {\bf harmonic numbers} $H_n$ are defined by $H_n = \sum_{k = 1}^n \frac{1}{k}$ and are interpolated by the complex function $H_z = \Psi(z+1)+ \gamma$,
where $\gamma$ is the Euler--Mascheroni constant and $\Psi(z) = \frac{\Gamma'(z)}{\Gamma(z)}$ is the {\bf digamma function}.  It is well known that $H_z -\gamma -\log z \sim \frac{1}{2z} \ (z \to \infty).$
Moreover, one has $H_z - \gamma = \Psi(z+1) = \Psi(z) + \frac{1}{z}$ for all $z \in \CC \backslash \{0,-1,-2,-3,\ldots \}$.  Thus,  one has
\begin{align}\label{harm}
H_x -\gamma = \log x+ o((\log x)^{-t}) \ (x \to \infty),
\end{align}
and likewise $\Psi(x) = \log x + o((\log x)^{-t}) \ (x \to \infty)$, for all $t > 0$.
\item The famous {\it Mertens' theorems} were proved by Mertens in 1874, over two decades before the first proofs of the prime number theorem.  The third of Mertens' theorems states that
$e^{\gamma} \prod_{p \leq x} \left(1-\frac{1}{p} \right)  \sim \frac{1}{\log x} \ (x \to \infty),$ where $e^\gamma  = \lim_{n \to \infty} \frac{e^{H_n}}{n} = 1.78107241\ldots.$  Given Mertens' theorem, the prime number theorem is equivalent to $\PP(x) \sim e^\gamma \prod_{p \leq x} \left(1-\frac{1}{p} \right) \ (x \to \infty)$.   Mertens' theorem can be improved to show that the function $L(x) = e^{-\gamma}\prod_{p \leq x}\left(1-\frac{1}{p}\right)^{-1}$  satisfies $L(x) \sim \log x + o ((\log x)^{-t}) \  (x \to \infty)$ for all $t > 0$ \cite[Lemma 5]{lan}.
\end{enumerate}
\end{example}

\begin{example}
Generalizing the obvious example $\frac{\pi(x)}{x}$,  Theorem \ref{arithsemigroup} of Section 5 provides examples $ \frac{\pi_G(x^{1/\delta})}{x}$ and $A\PP_G(x^{1/\delta})$ of  functions $P(x)$ that satisfiy the hypothesis $P(x) \sim \frac{\li(x)}{x} + o ((\log x)^{-t}) \  (x \to \infty)$, for all $t > 0$, of Proposition \ref{LP}.
\end{example}


From Proposition \ref{LP} and (\ref{harm}), we obtain the following.

\begin{corollary}\label{pas2}
One has the following asymptotic continued fraction expansions.
\begin{enumerate}
\item $\displaystyle \PP(x)\,  \simeq \,   \cfrac{\frac{1}{H_x-\gamma}}{1 \, -} \  \cfrac{\frac{1}{H_x-\gamma}}{1 \,-}\ \cfrac{\frac{1}{H_x-\gamma}}{1 \,-}\  \cfrac{\frac{2}{H_x-\gamma}}{1 \,-}\  \cfrac{\frac{2}{H_x-\gamma}}{1 \,-}  \ \cfrac{\frac{3}{H_x-\gamma}}{1 \,-}\  \cfrac{\frac{3}{H_x-\gamma}}{1 \,-} \  \cdots \ (x \to \infty).$
\item $\displaystyle  \PP(x) \, \simeq \,  \frac{1}{H_x-\gamma-1 \,-} \ \frac{1}{H_x-\gamma-3 \,-} \  \frac{4}{H_x-\gamma-5 \,-}\  \frac{9}{H_x-\gamma-7 \,-}\  \cdots \ (x \to \infty)$.
\item $\displaystyle \PP(e^\gamma x)\,  \simeq \,   \cfrac{\frac{1}{H_x}}{1 \, -} \  \cfrac{\frac{1}{H_x}}{1 \,-}\ \cfrac{\frac{1}{H_x}}{1 \,-}\  \cfrac{\frac{2}{H_x}}{1 \,-}\  \cfrac{\frac{2}{H_x}}{1 \,-}  \ \cfrac{\frac{3}{H_x}}{1 \,-}\  \cfrac{\frac{3}{H_x}}{1 \,-} \  \cdots \ (x \to \infty).$
\item $\displaystyle  \PP(e^\gamma x) \, \simeq \,  \frac{1}{H_x-1 \,-} \ \frac{1}{H_x-3 \,-} \  \frac{4}{H_x-5 \,-}\  \frac{9}{H_x-7 \,-}\  \frac{16}{H_x-9 \,-}\ \cdots \ (x \to \infty)$.
\end{enumerate}
\end{corollary}


\subsection{Asymptotic expansions related to some combinatorial integer sequences}

In this section we derive some analogues of the asymptotic expansion (\ref{asex2}) of $\PP(x)$ for various combinatorial sequences related to the sequence of factorials.

The following lemma is a straightforward application of the binomial theorem.

\begin{lemma}\label{asympprop2b}
Let $a_n, c_n \in \CC$ for all $n \geq 0$, let $f$ be a complex function, and let $\XX$ be an unbounded subset of $\CC$.  The asymptotic expansion  $f(x) \simeq \sum_{n = 0}^\infty \frac{a_n}{x^n} \ (x \to \infty)_\XX$ of $f$ with respect to $\left\{\frac{1}{x^n} \right\}$ is equivalent to the asymptotic expansion
$$f(x) \simeq \sum_{n = 0}^\infty \left( \sum_{k = 0}^n {n \choose k} a_k c_k^{n-k} \right) \! \frac{1}{(x+c_n)^n} \ (x \to \infty)_\XX$$
of $f$  with respect to $\left\{\frac{1}{(x+c_n)^n} \right\}$.
\end{lemma}

If  the sequence $c_n = t$ is constant, then the  sequence $b_n = \sum_{k = 0}^n {n \choose k} a_k t^{n-k}$ is called the {\bf $t$-binomial transform} of the sequence $a_n$.  The $0$-binomial transform of the sequence $n!$ is the sequence $n!$, which is equal to the number of permutations of any $n$-element set.  The $(-1)$-binomial transform of the sequence $n!$ is the sequence $D_n = n!\sum_{k = 0}^n \frac{(-1)^k}{k!}$, which is equal to the number of derangements of any $n$-element set.   Likewise, the $1$-binomial transform of the sequence $n!$ is the sequence $A_n = n! \sum_{k = 0}^n \frac{1}{k!}$, which is equal to the number of arrangements of any $n$-element set. Thus, an application of the  lemma to (\ref{asex2}) yields the following.

\begin{proposition}\label{basicp}
One has the following asymptotic expansions.
\begin{enumerate} 
\item $\displaystyle\PP(x) \simeq \sum_{n=0}^\infty \frac{D_{n}}{(\log x-1 )^{n+1}} \ (x \to \infty)$.
\item $\displaystyle\PP(ex) \simeq \sum_{n=0}^\infty \frac{D_{n}}{(\log x )^{n+1}} \ (x \to \infty)$.
\item $\displaystyle\PP(x) \simeq \sum_{n=0}^\infty \frac{A_{n}}{(\log x+1)^{n+1}} \ (x \to \infty)$.
\item $\displaystyle\PP(x/e) \simeq \sum_{n=0}^\infty \frac{A_{n}}{(\log x )^{n+1}} \ (x \to \infty)$.
\end{enumerate}
\end{proposition}

\begin{remark}
The two consequences below of Proposition \ref{basicp} demonstrate that combinatorial relationships between various sequences related to the factorials can yield interesting consequences for the asymptotic behavior of $\pi(x)$. 
\begin{enumerate}
\item The sequence $D_n$ begins $1,0,1,2,9,44,265,1854,\ldots$.   Since $D_1 = 0$ and $D_2 = 1$, the proposition provides yet another explanation for the fact that $\PP(x) - \frac{1}{\log x -1}$ is asymptotic to $\frac{1}{(\log x)^3}$.
\item The sequence  $A_n$  begins $1,2,5,16,65,326,1957,13700,\ldots$.    Thus one has
$$\PP(x/e) \simeq \frac{1}{\log x} +  \frac{2}{(\log x)^2} + \frac{5}{(\log x)^3} + \frac{16}{(\log x)^4} + \frac{65}{(\log x)^5}  + \cdots \ (x \to \infty).$$
At the same time, squaring the asymptotic expansion $\PP(x) \simeq \sum_{k = 0}^\infty \frac{n!}{(\log x)^{n+1}}$ of $\PP(x)$ yields
$$\PP(x)^2 \simeq \frac{1}{(\log x)^2} +  \frac{2}{(\log x)^3} + \frac{5}{(\log x)^4} + \frac{16}{(\log x)^5} + \frac{64}{(\log x)^6}  + \cdots \ (x \to \infty),$$
It follows that $ \PP(x/e) - \PP(x)^2 \log x  \sim \frac{1}{(\log x)^5} \ (x \to \infty).$
Consequently, one has $\PP(x)^2 \log x  < \PP(x/e)$ for all sufficiently large $x$, which is equivalent to Ramanujan's famous inequality
$$\pi(x)^2 < \frac{ex}{\log x} \pi(x/e), \quad x \gg 0.$$
\end{enumerate}
\end{remark}

For every nonnegative integer $n$, we let $r_n(X)$ denote the monic integer polynomial
$$r_n(X) = n! \sum_{k = 0}^n \frac{X^k}{k!}  =  \sum_{k = 0}^n \frac{n!}{k!} {X^k}\in \ZZ[X].$$ 
For any $z \in \CC$, the sequence $r_n(z)$ is the $z$-binomial transform of the sequence $n!$, and one has $r_n(0) = n!$, $r_n(-1) = D_{n}$, and $r_n(1) = A_{n}$ for all $n$, so the family of sequences $r_n(z)$ interpolates those three sequences.   
(Incidentally, by \cite[(14.1.9)]{cuyt}, one also has $r_n(z) =   e^z \Gamma(n+1,z)$ for all $z \in \CC$.)
Thus, from Lemma \ref{asympprop2b} we obtain the following generalization of Proposition \ref{basicp}.

\begin{proposition}\label{chlem}
For all $t \in \RR$ one has the asymptotic expansion
$$\PP(e^{-t }x) \simeq \sum_{n=0}^\infty \frac{r_n(t)}{(\log x)^{n+1}} \ (x \to \infty).$$
\end{proposition}

Note that, since $r_n(X) \in \ZZ[X]$, one has $r_n(k) \in \ZZ$ for all $n$ and all $k \in \ZZ$.  The integer sequence $r_n(2)$ is OEIS sequence A010842, and thus, for example, $r_n(2)$ is the number of ways to split the set $\{1,2,\ldots,n\}$ into two disjoint subsets $S$ and $T$ and linearly order $S$ and then choose a subset of $T$.   Also, the integer sequence $r_n(-2)$ is OEIS sequence A000023.   Note also that $r_n(z) \sim   n! e^z =  r_n(0)e^z  \ (n \to \infty)$ for all $z \in \CC$.  Thus, for example, one has $D_n  \sim \,  n! e^{-1} \ (n \to \infty)$ and $A_n   \sim \,  n!e \ (n \to \infty),$ which are well-known asymptotics for the sequences $D_n$ and $A_n$.

\begin{remark}
We provide a measure-theoretic proof of Proposition \ref{chlem}, and corresponding interpretations of the $t$-binomial transform and the polynomials $r_n(X)$, as follows.  Let $\mu$ be a measure on $\RR$.  For any fixed $t \in \RR$, let $T_t(\mu)$ be the measure on $\RR$ such that $T_t(\mu)(A) = \mu(T_{-t}(A))$ for any $\mu$-measurable set $A$, where $T_t(x) = x+t$ for all $x \in \RR$.  The $n$th moment of $T_t(\mu)$ is 
$m_n(T_t(\mu)) = \int_{-\infty}^\infty u^n\, dT_t(\mu)(u) = \int_{-\infty}^\infty (u+t)^n\, d\mu (u) = \sum_{k = 0}^n {n \choose k} m_k(\mu)t^{n-k},$
that is, the moment sequence of $T_t(\mu)$ is the $t$-binomial transform of the moment sequence of $\mu$.  Moreover, if $\mu$ has a density function, say, $\rho$, then $T_t(\mu)$ has density function $\rho \circ T_{-t}$. Thus, if $\gamma_0$ denotes the exponential distribution on $[0,\infty)$, then, for any fixed $t \in \RR$, the measure $T_t(\gamma_0)$ is the  probability measure on $[t,\infty)$ of density $e^{t-u} \, du$, and the $n$th moment of $T_t(\gamma_0)$ is given by  $m_n(T_t(\gamma_0)) = r_n(t)$, while also
$m_n(T_t(\gamma_0))  = \int_t^\infty  u^n e^{t-u} \, du  = e^t \Gamma(n+1, t)$.
Proposition \ref{chlem} is then equivalent to the obvious restatement $\PP(e^{-t }x) \simeq \sum_{n=0}^\infty \frac{m_n(T_t(\gamma_0))}{(\log x)^{n+1}} \ (x \to \infty)$ of (\ref{asex2}).
\end{remark}

For any positive integer $n$, let $D_n'$ denote the number of indecomposable derangements of $\{1,2,3,\ldots, n\}$, and let $D_0' = -1$.    The sequence $\{D_n'\}$ is OEIS sequence A259869 and begins $-1,0, 1, 2, 8, 40, 244, 1736,  \ldots$.  It has generating function $\sum_{n = 0}^\infty D_n' z^n = -\frac{1}{\sum_{n = 0}^\infty D_n z^n}.$  Proposition \ref{basicp}(2) therefore has the following corollary.

\begin{corollary}
One has the asymptotic expansion
$$\frac{1}{\PP(ex)} \simeq - \sum_{n=0}^\infty \frac{D_{n}'}{(\log x )^{n-1}} \ (x \to \infty).$$
\end{corollary}

We can  also invert the expansion  $\frac{\li(x)}{x} \simeq \sum_{n = 0}^\infty \frac{n!}{(\log x)^{n+1}} \ (x \to \infty)$ 
as one does formal power series under composition.  Let $G(z) = - \sum_{n= 1}^\infty G_n z^n$ denote  the series inversion of the formal power series $F(z) = \sum_{n = 0}^\infty n! z^{n+1},$
that is, let $G(z)$ denote the unique formal power series  such that  $z = G(F(z)) = F(G(z)).$
One has $G_1 = -1$, and the sequence $\{G_n\}_{n = 2}^\infty$ appears as OEIS sequence A134988, with first several terms given by  $1, 0, 1, 4, 22, 144, 1089, 9308,  \ldots$.  In particular, $G_n$ for any $n \geq 2$ is the number of generators in arity $n$ of the free non-symmetric operad $\mathfrak{Lie}$ \cite{salv}, and $G_n$ for all $n \geq 2$ also has a de Rham cohomological interpretation as $\dim_\QQ H^{n-2}({\mathfrak M}^\delta_{0,n+1}, \QQ)$, where ${\mathfrak M}^\delta_{0,n}$ for any $n \geq 3$ is a certain smooth affine scheme that approximates the moduli space ${\mathfrak M}_{0,n}$, defined over $\ZZ$, of smooth $n$-pointed curves of genus 0  \cite{brow}.    It is also known that $\frac{z}{G(z)} = \sum_{n = 0}^\infty J_nz^n,$
where $J_n$ is the number of {\it stabilized-interval-free} permutations of $\{1, 2, \ldots , n\}$ and is OEIS sequence A075834, with first several terms given by $1,1,1,2,7,34, 206,1476, \ldots$ \cite{call}.    The definitions of the sequences $G_n$ and $J_n$ yield  statements (4) and (5) of the following proposition, which summarizes the asymptotic expansions proved in this section and in Section 1.2.

\begin{proposition}\label{LP2}
Let $L(x)$ be any function such that $L(x) = \log x + o ((\log x)^{-t}) \  (x \to \infty)$ for all $t > 0$, and let $P(x)$ be any function such that $P(x) = \frac{\li(x)}{x} + o ((\log x)^{-t}) \  (x \to \infty)$ for all $t > 0$.   One has the following asymptotic expansions.
\begin{enumerate}
\item $\displaystyle  P(e^{-t }x) \simeq \sum_{n=0}^\infty \frac{r_n(t)}{L(x)^{n+1}} \ (x \to \infty)$ for all $t \in \RR$, where $r_n(t) = \sum_{k = 0}^n \frac{n!}{k!} {t^k}$, $r_n(0) = n!$, $r_n(1) = A_n$, and $r_n(-1) = D_n$ for all $n$.
\item $\displaystyle L(x) - \frac{1}{P(x)} \simeq  \sum_{n = 0}^{\infty} {\frac {I_{n+1}}{L(x)^{n}}} \ (x \to \infty).$
\item $\displaystyle  \frac{1}{P(ex)} \simeq - \sum_{n = 0}^{\infty} {\frac {D_{n}'}{L(x)^{n-1}}} \ (x \to \infty).$
\item $\displaystyle \frac{1}{L(x)} \simeq -\sum_{n = 1}^\infty G_n P(x)^n \ (x \to \infty)$.
\item $\displaystyle {L(x)} \simeq   \sum_{n = -1}^\infty J_{n+1} P(x)^n \ (x \to \infty)$.
\end{enumerate}
\end{proposition}

\section{Generalizations to arithmetic semigroups and number fields}

All of the asymptotic expansions of $\PP(x)$ we have proved in this paper can be generalized to any number field  $K$ as follows.   For all $x > 0$, let $\pi_K(x)$ denote the number of prime ideals of ${\mathcal O}_K$ of norm less than or equal to $x$.  The {\bf Landau prime ideal theorem}, proved by Landau in 1903, states that 
$$\pi_K(x) = \li(x) + O\left(xe^{-c_K {\sqrt  {\log x }}}\right)\ (x \to \infty)$$
for some constant $c_K > 0$ depending on $K$.  Landau's  theorem implies that
$$\frac{\pi_K(x)}{x} = \frac{\li(x)}{x} + O((\log x)^{-t}) \ (x \to \infty)$$
for all $t > 0$.   It follows that the function  $\frac{\pi_K(x)}{x}$ satisfies all of the properties and asymptotic expansions  that we have proved for the function $\PP(x)$ (and, for example, it satisfies the property required for the functions $P(x)$ in Propositions \ref{LP} and \ref{LP2}).

In fact, we may extend our analysis to any arithmetic semigroup $G$ that satisfies a certain generalization of Axiom A \cite{kno}.  An {\bf arithmetic semigroup} is a commutative monoid $G$ (written multiplicatively) that is freely generated by a countable set $P$, called the set of {\bf primes} of $G$, and that is equipped with a map $|\cdot|: G \longrightarrow \RR_{>0}$, called the {\bf norm} map of $G$, such that the following conditions hold.
\begin{enumerate}
\item  $|1| = 1$.
\item $|p| > 1$ for all $p \in P$.
\item  $|ab| = |a||b|$ for all $a,b \in G$. 
\item  For all $x > 0$, the number $N_G(x) = \#\{a \in G: |a| \leq x\}$ of elements of $G$ of norm at most $x$ is finite.
\end{enumerate}
In fact, condition (4) can be replaced with
\begin{enumerate}
\item[$4'.$]  For all $x > 0$, the number $\pi_G(x) = \#\{p \in P: |p| \leq x\}$ of primes of $G$ of norm at most $x$ is finite.
\end{enumerate} 
For all $x \geq 1$, we let $\PP_G(x) = \frac{\pi_G(x)}{N_G(x)}$ denote the probability that a randomly selected element of $G$ of norm less than or equal to $x$ is prime.

 Let $G$ be any  arithmetic semigroup.   For any $\delta > 0$, we let $G_\delta$ denote the commutative monoid $G$ together with norm map $|\cdot|^{\delta}: G \longrightarrow \RR_{>0}$, where $|\cdot|$ is the norm map of the arithmetic semigroup $G$.  It is clear that $G_\delta$ is also an arithmetic semigroup, and one has
$$N_{G_\delta}(x) =  \#\{a \in G: |a|^\delta \leq x\} =  \#\{a \in G: |a| \leq x^{1/\delta}\} = N_G(x^{1/\delta}),$$
and likewise $\pi_{G_\delta}(x)  =  \pi_G(x^{1/\delta}),$
for all $x > 0$.   Suppose that there exist constants  $\delta > 0$ and $A > 0$ such that
\begin{align*}
N_G(x) = Ax^\delta+ O(x^\delta (\log x)^{-t}) \ (x \to \infty)
\end{align*}
for all $t > 0$.  By \cite[p.\ 152]{weg}, one then has
\begin{align*}
\pi_G(x) = \li(x^\delta) + O(x^\delta (\log x)^{-t}) \ (x \to \infty)
\end{align*}
for all $t > 0$. 
Moreover, by \cite[p. 304]{ny}, one has $A \delta = \operatorname{Res}_{s = \delta} \zeta_G(s) = \lim_{\operatorname{Re}{s} \to \delta^+} (s-\delta) \zeta_G(s),$ where the series defining the {\bf zeta function $\zeta_G(s) = \sum_{a \in G} \frac{1}{|a|^{s}}$ of $G$} is absolutely convergent for $\operatorname{Re}(s) > \delta$ and has meromorphic continuation to $\operatorname{Re}(s) \geq \delta$ with a unique (simple) pole at $s = \delta$.  For the proofs of these results,  by replacing $G$ with $G_\delta$,  one may assume  without loss of generality that $\delta = 1$ (which is assumed throughout \cite{ny}).  

Using  Corollary \ref{wsim}, Lemma \ref{asympprop}, (\ref{asex}), and the results mentioned above, we can generalize Theorem \ref{maincontthm1} and Corollary \ref{maincor} as follows.

\begin{theorem}\label{arithsemigroup}
Let $G$ be an arithmetic semigroup, and let $\delta > 0$.   Let $w_n(x)$ for any nonnegative integer $n$ denote the $n$th approximant of the continued fraction
$$\frac{1}{x-1 \,-} \  \frac{1}{x-3 \,-} \  \frac{4}{x-5 \,-}\  \frac{9}{x -7\,-} \ \frac{16}{x-9 \,-}\  \cdots.$$
\begin{enumerate}
\item The following conditions are equivalent. 
\begin{enumerate}
\item For all $t > 0$, one has $\pi_G(x) = \li(x^\delta) + O(x^\delta (\log x)^{-t}) \ (x \to \infty).$
\item $\frac{\pi_G(x)}{x^\delta}$ has the asymptotic expansion
$\frac{\pi_G(x)}{x^\delta} \, \simeq \, \sum_{k = 0}^\infty \frac{k!}{(\delta \log x)^{k+1}} \ (x \to \infty).$
\item $\pi_G(x)$ has the asymptotic continued fraction expansion
 \begin{align*}
{\pi_G(x)} \, \simeq \, \cfrac{\frac{x^\delta}{\delta \log x}}{1 \,-} \  \cfrac{\frac{1}{\delta \log x}}{1 \,-} \  \cfrac{\frac{1}{\delta \log x}}{1 \,-}\  \cfrac{\frac{2}{\delta \log x}}{1 \,-}\  \cfrac{\frac{2}{\delta \log x}}{1 \,-} \  \cfrac{\frac{3}{\delta \log x}}{1 \,-}\  \cfrac{\frac{3}{\delta \log x}}{1 \,-} \ \cdots \ (x \to \infty).
\end{align*}
\item $\pi_G(x)$ has the asymptotic continued fraction expansion 
\begin{align*}
\pi_G(x) \, \simeq \, \frac{x^\delta}{\delta \log x-1 \,-} \  \frac{1}{\delta \log x-3 \,-} \  \frac{4}{\delta\log x-5 \,-}\  \frac{9}{\delta\log x-7 \,-} \ \frac{16}{\delta\log x-9 \,-}\  \cdots  \ (x \to \infty).
\end{align*}
\item $\frac{\pi_G(e^{x/\delta})}{e^x}-w_n(x) \sim \frac{(n!)^2}{x^{2n+1}} \ (x \to \infty)$ for all nonnegative integers $n$.
\item $\frac{\pi_G(e^{x/\delta})}{e^x}-w_n(x) =  O\left( \frac{1}{x^{2n+1}}\right) \ (x \to \infty)$ for all nonnegative integers $n$.
\item $w_n(x)$ for every nonnegative integer $n$ is the unique rational function $w(x) \in \CC(x)$ of degree at most $n$ such that $\frac{\pi_G(e^{x/\delta})}{e^x}-w_n(x) = O \left( \frac{1}{x^{2n+1}}\right) \ (x \to \infty)$. 
\end{enumerate}
\item Suppose that there exists a constant $A > 0$ such that
\begin{align}\label{Nhyp}
N_G(x) = Ax^\delta+ O(x^\delta (\log x)^{-t}) \ (x \to \infty)
\end{align}
for all $t > 0$.  Then $A \delta = {\operatorname{Res}_{s = \delta} \zeta_G(s)}$, the equivalent conditions (1)(a)--(g) hold, and one has the asymptotic continued fraction expansions
 \begin{align*}
\PP_G(x) \, \simeq \, \cfrac{\frac{1}{A\delta \log x}}{1 \,-} \  \cfrac{\frac{1}{\delta \log x}}{1 \,-} \  \cfrac{\frac{1}{\delta \log x}}{1 \,-}\  \cfrac{\frac{2}{\delta \log x}}{1 \,-}\  \cfrac{\frac{2}{\delta \log x}}{1 \,-} \  \cfrac{\frac{3}{\delta \log x}}{1 \,-}\  \cfrac{\frac{3}{\delta \log x}}{1 \,-} \ \cdots  \ (x \to \infty),
\end{align*}
and
\begin{align*}
\PP_G(x)\,  \simeq \, \frac{1/A}{\delta \log x-1 \,-} \  \frac{1}{\delta \log x-3 \,-} \  \frac{4}{\delta\log x-5 \,-}\  \frac{9}{\delta\log x-7 \,-} \ \frac{16}{\delta\log x-9 \,-}\  \cdots  \ (x \to \infty),
\end{align*}
Moreover, the $w_n(x)$ for all nonnegative integers $n$ are precisely the best rational approximations of the functions $\frac{\pi_G(e^{x/\delta})}{e^x}$ and $A\PP_G(e^{x/\delta})$.
\end{enumerate}
\end{theorem}

\begin{proof}
It is clear that condition (1)(a) holds if and only if  $\frac{\pi_{G_\delta}(x)}{x} =  \frac{\li(x)}{x} + O( (\log x)^{-t}) \ (x \to \infty)$ for all $t > 0$.   Extending such an analysis to the entire statement of the theorem, we may assume without loss of generality, by replacing $G$ with $G_\delta$, that $\delta = 1$.  By (\ref{asex}) and Lemma \ref{asympprop}, conditions (1)(a) and (1)(b) are equivalent.   Moreover, conditions (1)(b)--(g) are equivalent by Corollary \ref{wsim} and the proof  in Section 3.1 of  Theorem \ref{maincontthm1}.  This proves statement (1).  Assuming the hypothesis of statement (2) (with $\delta = 1$), condition (1)(a) then follows from    \cite[p.\ 152]{weg}, one has $A  = {\operatorname{Res}_{s = 1} \zeta_G(s)}$ by \cite[p.\ 304]{ny},  and one has $A\PP_G(x)=  \frac{\pi_G(x)}{N_G(x)/A} =   \frac{\li(x)}{x} + O((\log x)^{-t}) \ (x \to \infty)$ for all  $t > 0$.  Therefore, as with statement (1), statement (2) follows from  (\ref{asex}), Lemma \ref{asympprop}, and Corollary \ref{wsim}.
\end{proof}

An arithmetic semigroup $G$ satisfies {\bf Axiom A}  if there exist positive constants $A$ and $\delta$ and a nonnegative constant $\eta < \delta$ such that  $N_G(x) = Ax^\delta+O(x^\eta) \ (x \to \infty)$  \cite[Chapter 4]{kno}.  Clearly, if $G$ is an arithmetic semigroup satisfying Axiom A, then the hypothesis (\ref{Nhyp}) of statement (2) of Theorem \ref{arithsemigroup} holds.  Thus, the various conditions in the theorem hold for any arithmetic semigroup satisfying Axiom A.   \cite[Chapters 4 and 5]{kno} provides a wide variety of interesting examples of arithmetic semigroups satisfying Axiom A, including the following defined for any number field.

\begin{example}[{\cite{kno}}]  Let $K$ be a number field.
\begin{enumerate}
\item Let $G$ be the monoid, under ideal multiplication, of all nonzero ideals of $\mathcal{O}_K$.  Then $G$ is an arithmetic semigroup satisfying Axiom A, with the norm map acting by $\mathfrak{a} \longmapsto N_{K/\QQ}(\mathfrak{a}) = \#( \mathcal{O}_K/\mathfrak{a})$, and with $\delta = 1$ and $\zeta_G(s) = \zeta_K(s)$, so that $A = \operatorname{Res}_{s = 1} \zeta_K(s)$, where $\zeta_K(s) = \sum_{\mathfrak{a} \in G} \frac{1}{(N_{K/\QQ}(\mathfrak{a}))^s}$ is the Dedekind zeta function of $K$.  Moreover, $\PP_G(x)$ for all $x \geq 1$ is the probability that a randomly selected ideal of ${\mathcal O}_K$ of norm less than or equal to $x$ is prime.  For example, if $K  = \QQ(i)$, then $A = \pi/4$, and if $K = \QQ(\sqrt{2})$, then $A =\frac{1}{\sqrt{2}}\log(1+\sqrt{2})$.  Also, if $K = \QQ$, then $\zeta_G(s) = \zeta(s)$ is the Riemann zeta function, $A = 1$, and $\PP_G(x) =  \frac{\pi(x)}{\lfloor x \rfloor}$.
\item Let $G$ be the monoid,  under direct product, of all isomorphism classes of $\mathcal{O}_K$-modules of finite cardinality.  Then $G$ is an arithmetic semigroup satisfying Axiom A, with the norm map acting by $M \longmapsto \#M$, and with $\delta = 1$ and $\zeta_G(s) = \prod_{n =1}^\infty \zeta_K(ns)$, so that $A = \operatorname{Res}_{s = 1} \zeta_K(s) \prod_{n = 2}^\infty \zeta_K(n)$.  Moreover, $\PP_G(x)$ for all $x \geq 1$ is the probability that an $\mathcal{O}_K$-module of cardinality less than or equal to $x$, randomly selected among a set of representatives of the isomorphism classes, is indecomposable. 
\end{enumerate}
\end{example}


\begin{thebibliography}{}

\bibitem{akh} N.\ I.\  Akhiezer, {\it The Classical Moment Problem: and Some Related Questions in Analysis}, University Mathematical Monographs, Oliver \& Boyd, London, 1965.



 


\bibitem{brow} F.\ Brown and J.\ Bergstr\"om, Inversion of series and the cohomology of the moduli spaces ${\mathcal M}^\delta_{0,n}$, in {\it Motives, quantum field theory, and pseudodifferential operators}, Clay Math.\ Proc.\ 12 (2010) 119-126, Amer.\ Math.\ Soc., Providence.


\bibitem{call} D.\ Callan, Counting stabilized-interval-free permutations, J.\ Integer Sequences 7 (2004) Article 04.1.8.

\bibitem{cheb}  P.L.\  Chebyshev, Sur la fonction qui d\'etermine la totalit\'e des nombres premiers inf\'erieurs \`a une limite donn\'ee, Mem.\ pres.\ Acad.\ Imp.\ Sci.\ St.\ Petersb.\ 6 (1851) 141--157.


\bibitem{comt}  L.\ Comtet,  Sur les coefficients de l'inverse de la s\'erie formelle $\sum n! t^n$, C.\ R.\ Acad.\ Sci.\ Paris  A 275 (1972) 569--572.

\bibitem{copson}  E.T.\ Copson, {\it Asymptotic Expansions}, Cambridge Tracts in Mathematics 55, Cambridge University Press, 1965.

\bibitem{cuyt} A.\ A.\ M.\ Cuyt, V.\ Petersen, B.\ Verdonk, H.\ Waadeland, and W.\ B.\ Jones, {\it Handbook of Continued Fractions for Special Functions}, Springer Science+Business Media, 2008.






\bibitem{erd}  A.\ Erd\'elyi, {\it Asymptotic Expansions}, Dover Publications, Inc., 1956.

\bibitem{est}  R.\ Estrada and R.P.\ Kanwal, {\it Asymptotic Analysis: A Distributional Approach}, Birkh\"auser, Boston, 1994.

\bibitem{val1} C.-J.\ de la Vall\'ee Poussin, Recherches analytiques la th\'eorie des nombres premiers,  Ann.\ Soc.\ scient.\ Bruxelles 20 (1896) 183--256.

\bibitem{val2}  C.-J.\ de la Vall\'ee Poussin, Sur la fonction Zeta de Riemann et le nombre des nombres premiers inferieur a une limite donn\'ee, C.\ M\'em.\ Couronn\'es Acad.\ Roy.\ Belgique 59 (1899) 1--74.


\bibitem{gau}  C.\ F.\ Gauss, {\it Werke}, Bd.\ 2, First Edition, G\"ottingen, 1863, pp.\ 444--447.






\bibitem{had} J.\ Hadamard,  Sur la distribution des z\'eros de la fonction $\zeta(s)$ et ses cons\'equences arithm\'etiques,  Bull.\ Soc.\ Math.\ France 24 (1896) 199--220.

\bibitem{ham}  H.\ Hamburger, \"Uber eine Erweiterung des Stieltjesschen Momentenproblems, Math.\
Ann.\ 81 (1920) 235--319; 82 (1921) 120--164, 168--187.


\bibitem{hall}  M. Hall, Jr., Subgroups of finite index in free groups, Can.\  J.\  Math.\ 1 (1949)
187--190.


\bibitem{kno}  J.\ Knopfmacher, {\it Abstract Analytic Number Theory}, Second Edition, Dover Publications, Inc., New York, 1990.


\bibitem{koch}  H.\ von Koch, Sur la distribution des nombres premiers, Acta Math.\ 24 (1901) 159--182.


\bibitem{lag}    E.N.\ Laguerre, Sur l'int\'egrale $\int_x^\infty \frac{e^{-x} dx}{x}$, Bull.\ de la Soc.\ Math.\ de France 7 (1879) 72--81.

\bibitem{lan}  A.\ Languasco and A.\ Zaccagnini, A note on Mertens' formula for arithmetic progressions, J.\ Number Theory 127 (2007) 37-46.


\bibitem{lor}  L.\ Lorentzen and H.\ Waadeland,  {\it Continued Fractions with Applications}, Studies in Computational Mathematics 3, North-Holland, Amersterdam, 1992.











\bibitem{ny} B.\ Nyman, A general prime number theorem, Acta Math.\ 81 (1949) 299--307.


\bibitem{pan} L.\ Panaitopol, A formula for $\pi(x)$ applied to a result of Konick-Ivi\'c, Nieuw Arch.\ Wiskd.\ (5) 1 (1) (2000)  55--56.


\bibitem{PSZ} A.\ Poltoratski, B.\ Simon, and M.\ Zinchenko, The Hilbert transform of a measure, J.\ Anal.\ Math.\ 111 (1) (2010) 247--265.


\bibitem{salv}  P.\ Salvatore and R.\ Tauraso, The operad Lie is free, J.\ Pure Appl.\ Algebra 213 (2) (2009) 224--230.

\bibitem{stop}  J.\ Stopple, {\it A Primer of Analytic Number Theory: From Pythagoras to Riemann}, Cambridge University Press, 2003.



\bibitem{sko} H.\ Skovgaard, On the greatest and the least zero of Laguerre polynomials, Matematisk Tidsskrift.\ B, 1951, tematisk tidsskrift.\ B (1951) 59--66.

\bibitem{stie} T.J.\ Stieltjes, Recherches sur les fractions continues, Ann.\ Fac.\ Sci.\ Toulouse S\'er.\ 1, 8 (4) (1894) J1--J122.



\bibitem{wall}  H.S.\ Wall, {\it Analytic Theory of Continued Fractions}, University Series in Higher Mathematics, Van Nostrand, New York, 1948.

\bibitem{weg} H.\ Wegmann, Beitr\"age zur Zahlentheorie auf freien Halbgruppen, I--II, J.\ Reine Angew.\ Math.\ 221 (1965) 20--43, 150--159.
 

\end{thebibliography}
\end{document}